\documentclass[reqno]{amsart}

\usepackage{amssymb,amsfonts,amsmath,amsthm}
\usepackage{mathrsfs}
\usepackage{graphicx}
\usepackage{color}
\usepackage{verbatim}
\usepackage{wasysym}
\usepackage{mathtools}

\newtheorem{thm}{Theorem}[section]
\newtheorem{pro}[thm]{Proposition}
\newtheorem{lem}[thm]{Lemma}
\newtheorem{cor}[thm]{Corollary}

\theoremstyle{definition}

\theoremstyle{remark}
\newtheorem{rmk}[thm]{Remark}

\newtheorem{prb}{Problem}

\numberwithin{equation}{section}

\def\J{\mathscr{J}}
\def\D{\mathscr{D}}
\def\R{\mathscr{R}}
\def\L{\mathscr{L}}
\def\H{\mathscr{H}}
\def\cE{\mathcal{E}}

\def\es{\varnothing}
\def\Ra{\Rightarrow}

\def\lra{\leftrightarrow}
\def\ol#1{\overline{#1}}

\def\ig#1{\mathsf{IG}(#1)}

\def\pre#1#2{\langle #1 \; | \; #2 \rangle}
\renewcommand{\le}{\leqslant}

\makeatletter
\def\mathcenterto#1#2{\mathclap{\phantom{#1}\mathclap{#2}}\phantom{#1}}
\let\old@widetilde\widetilde
\def\widetildeto#1#2{\mathcenterto{#2}{\old@widetilde{\mathcenterto{#1}{#2\,}}}}
\def\widetilde{\widetildeto{||||}}
\def\tR{\widetilde{\mathscr{R}}}
\def\tL{\widetilde{\mathscr{L}}}

\DeclareMathOperator\Rat{Rat} 
 
\DeclareMathOperator\fix{fix} 

\begin{document}


\title[Free idempotent generated semigroups]%
{A group-theoretical interpretation of the word problem for free idempotent generated semigroups} 


\author{YANG DANDAN}

\address{School of Mathematics and Statistics, Xidian University, Xi'an 710071, P.R.China}

\email{ddyang@xidian.edu.cn}

\author{IGOR DOLINKA}

\address{Department of Mathematics and Informatics, University of Novi Sad, Trg Dositeja Obradovi\'ca 4,
21101 Novi Sad, Serbia}

\email{dockie@dmi.uns.ac.rs}

\thanks{The research of the second-named author was supported by the Ministry of Education, Science, and Technological
Development of the Republic of Serbia through the Grant No.174019.}

\author{VICTORIA GOULD}

\address{Department of Mathematics, University of York, Heslington, York YO10 5DD, UK}

\email{victoria.gould@york.ac.uk}









\begin{abstract}
The set of idempotents of any semigroup carries the structure of a biordered set, which contains a great deal of
information concerning the idempotent generated subsemigroup of the semigroup in question. This leads to the construction of a
free idempotent generated semigroup $\ig{\cE}$ -- the `free-est' semigroup with a given biordered set $\cE$ of idempotents.
We show that when $\cE$ is finite, the word problem for $\ig{\cE}$ is equivalent to a family of constraint satisfaction problems 
involving rational subsets of direct products of pairs of maximal subgroups of $\ig{\cE}$. As an application, we obtain decidability
of the word problem for an important class of examples. Also, we prove that for finite $\cE$, $\ig{\cE}$ is always a weakly abundant
semigroup satisfying the congruence condition.
\end{abstract}


\maketitle



\section{Introduction}
\label{sec:intro}

Let $S$ be a semigroup, $E=E(S)$ the set of its idempotent elements, and let $\langle X\rangle$ denote the subsemigroup of $S$
generated by $X\subseteq S$. We say that $S$ is \emph{idempotent generated} if $S=\langle E\rangle$, that is, if every element
of $S$ can be written as a product of idempotents. Idempotent generated semigroups are ubiquitous in algebra. For example, it was
shown in \cite{Er,La} that every singular matrix over a division ring $Q$ is a product of idempotent matrices over $Q$; hence,
the idempotent generated subsemigroup of the full linear monoid $M_n(Q)$ coincides with the submonoid consisting of the singular
matrices and the identity matrix (basically, $M_n(Q)$ with its group of units $GL_n(Q)$ `torn off' and the identity reinstated). An analogous result is
proved by Howie \cite{Ho} for the transformation monoid $\mathcal{T}_n$ of all self-maps of a finite set with $n$ elements: every non-bijective
self-map of $\{1,\dots,n\}$ is a composition of idempotent transformations. In the same paper it was shown that every semigroup embeds into
an idempotent generated one, thus highlighting the importance of idempotent generated semigroups. Among numerous contributions
in this vein, let us also mention the more recent work of Putcha \cite{Pu} who fully characterised the reductive linear
algebraic monoids \cite{PuBook,ReBook} with the property that each non-unit is a product of idempotents.

Another key departure point for this paper is the extensive study of the structure of regular semigroups \cite{Na} by Nambooripad,
who made the crucial observation that the set of idempotents of any semigroup carries a certain `geometric' structure, called
the \emph{biordered set}, that is intimately related to the global structural properties of the idempotent generated subsemigroup of
the semigroup in question. He supplied an axiomatic approach to (regular) biordered sets. Easdown
later showed \cite{Ea2}
that these axioms are complete in the sense that every abstractly defined biordered set is isomorphic to a biordered set of
some semigroup (Nambooripad has proved this for regular biordered sets and regular semigroups). Therefore, there is no loss
in omitting the abstract definition and restricting our attention to biordered sets of concrete semigroups. So, for a semigroup $S$
and its set of idempotents $E$ we define the biordered set of idempotents of $S$ as the partial algebra
$$\cE_S=(E,\ast)$$ inheriting some of the products
from $S$ so that $e\ast f=ef$ whenever $e,f\in E$ is a \emph{basic pair}, that is, whenever $\{e,f\}\cap\{ef,fe\}\neq\es$. (Note that this condition implies that both $ef,fe$ are idempotents for a basic pair $e,f$: if,
for example, $ef=e$, then $(fe)^2=f(ef)e=fee=fe$.)
It is convenient to define two quasi-orders on $E$, $\le_r$ and $\le_\ell$, where $e\le_r f$ if and only if
$fe=e$ and $e\le_\ell f$ if and only if $ef=e$. Now $e,f$ form a basic pair if and only if
$\le_r\cup\le_\ell$ contains either $(e,f)$ or $(f,e)$. Furthermore, $\leqslant\; =\;
\le_r\cap\le_\ell$ is easily seen to be a partial order, and this is the \emph{natural partial order} on the idempotents $E$ of $S$,
while $\R$ and $\L$ will denote (in what will turn out to be a slight but welcome abuse of notation) the equivalences induced by
$\le_r$ and $\le_\ell$, respectively; finally, $\D=\R\vee\L$ is the least equivalence on $E$ containing both $\R$ and $\L$.
By \cite[Lemma 3.1]{DGR}, all these three equivalences are effectively computable from $\cE$ (provided $\cE$
is finite).

For a given biordered set $\cE=(E,\ast)$, there is a category of semigroups associated with $\cE$ as follows. 
The objects are pairs $(S,\phi)$ where $S$ is a semigroup and $\phi:\cE\to\cE_S$ is an isomorphism of biordered
sets. The morphisms from $(S,\phi)$ to $(T,\psi)$ are semigroup homomorphisms $\theta:S\to T$ such that $\phi\theta=\psi$
(which implies that the restriction of $\theta$ to $E(S)$ is a biordered set isomorphism $\cE_S\to\cE_T$). It transpires
that this category has an initial object $(\ig{\cE},\iota_E)$ (where $\iota_E$ is the identity mapping on $E$). The semigroup
$\ig{\cE}$ is called the \emph{free idempotent generated semigroup} over $\cE$. For reasons that will be explained shortly, 
we introduce a new set $\ol{E}=\{\ol{e}:\ e\in E\}$ in bijective correspondence with $E$, serving as the 
generating set of $\ig{\cE}$, and define $\ig{\cE}$ by the presentation
$$
\ig{\cE} = \langle\, \ol{E}:\ \ol{e}\,\ol{f} = \ol{e\ast f}\text{ whenever }\{e,f\}\text{ is a basic pair}\, \rangle .
$$
We introduce this slight notational twist because, in what follows, it will be of great importance for us to distinguish
between \emph{words} over $E$, elements of the free semigroup $E^+$ (which we write without bars), and \emph{elements}
of the semigroup $\ig{\cE}$. These latter elements clearly arise from words over $E$, in fact, $\ig{\cE}$ is a quotient
of $\ol{E}^+$ (and thus a homomorphic image of $E^+$), but with the bar-notation $\ol{w}$, $w\in E^+$, we want to put
emphasis on the fact that we are working with a particular element of $\ig{\cE}$ induced by the word $w$ and that we
are considering its position and properties within this latter semigroup. In particular, we denote by $\ol{w}
\in\ig{\cE}$ the image of $w\in E^+$ under the (natural) homomorphism $E^+\to\ig{\cE}$ and thus $\ol{u}\,\ol{v}=\ol{uv}$ 
holds for any $u,v\in E^+$.

Clearly, a crucial step in understanding the structure of idempotent generated semigroups -- and in particular any subclass
with a fixed biordered set of idempotents $\cE$ -- is to study the structure and properties of these free objects $\ig{\cE}$.
Recently, there has been a significant surge of interest towards this goal. A dominant aspect of the topic so far concerned
the maximal subgroups of semigroups of the form $\ig{\cE}$. It was long conjectured that the maximal subgroups of free idempotent
generated semigroups must be free groups, and this conjecture was eventually recorded in \cite{McE}. It was
refuted by Brittenham, Margolis, and Meakin \cite{BMM}, and in fact Gray and Ru\v skuc proved in \cite{GR1} that quite the
opposite is true: every group arises as a maximal subgroup of a suitable free idempotent generated semigroup. This latter
result was subsequently reproved in \cite{DR} and \cite{GY} by using biordered sets of very specific semigroups such as idempotent
semigroups (bands) and endomorphism monoids of free $G$-acts, respectively. Also, one thread of investigation was devoted to
computing the maximal subgroups of $\ig{\cE}$ for biorders of some of the most natural examples of idempotent generated semigroups
\cite{Do2,DG,GR2,YDG}, thus supplying further counterexamples to the freeness conjecture.

The focus changed substantially with the recent paper \cite{DGR}, which initiated investigation of the word problem for
free idempotent generated semigroups. The results from \cite{DGR} relevant to our goals will be reviewed in more detail
in the next section, but we briefly mention that it was shown that (a) there is an algorithm which, given a word $w\in E^+$,
decides whether $\ol{w}$ is a regular element of $\ig{\cE}$; (b) if the word problems of all the maximal subgroups of
$\ig{\cE}$ are solvable, then there is an algorithm  which, for given $u,v\in E^+$ such that $\ol{u},\ol{v}$ are regular
in $\ig{\cE}$, decides whether $\ol{u}=\ol{v}$; (c) there exist a finite (idempotent) semigroup $S$ such that the word
problem of $\ig{\cE_S}$ is undecidable, even though all of its maximal subgroups have decidable word problems. (This
was achieved by encoding the undecidability of the subgroup membership problem for direct squares of free groups, a classical
result due to Mihailova \cite{Mi}.)

The present paper follows the footsteps of \cite{DGR}, but setting a much more general goal. Namely, in a certain 
sense, we aim to `reveal the true nature' of the word problem of a semigroup of the form $\ig{\cE}$ where $\cE$ is a
finite biordered set by translating and recasting it as an algorithmic problem in group theory. We already mentioned that
the word problem of $\ig{\cE}$ restricted to its regular part boils down to word problems of its maximal subgroups (since
the results of \cite{DGR} also show that the converse of (b) above holds as well). We are going to show that the general
word problem for $\ig{\cE}$ is equivalent to a family of decision problems (in the form of existential primitive positive
formulae -- essentially, constraint satisfaction problems) involving rational subsets of direct products of pairs of
maximal subgroups of $\ig{\cE}$ that are effectively computable from $\cE$. This will be achieved in a number of steps.
\begin{itemize}
  \item First, in the next, preliminary, section we gather the necessary basic tools of semigroup theory and language 
	theory, as well as the key results from the literature needed for our considerations. 
  \item In Sect.\ 3 we undertake an analysis of the so-called \emph{r-factorisations} of a word $w\in E^+$, namely 
	factorisations of $w$ into subwords representing regular elements of $\ig{\cE}$. \textbf{Theorem \ref{D-fing}} proves 
	that in any minimal (coarsest) r-factorisation of $w$, the sequence of $\D$-classes to which the regular factors belong 
	(called the \emph{$\D$-fingerprint}) is an invariant of the word $w$. This result is key to many later arguments. 
	In particular, any two words $u,v\in E^+$ must have the same $\D$-fingerprint to stand any chance of satisfying 
	$\ol{u}=\ol{v}$ in $\ig{\cE}$.
	\item Then, in Sect.\ 4, following a crucial idea from \cite[Section 4]{DGR}, with the aim of deciding $\ol{u}=\ol{v}$ 
	we rewrite (again, in an effective, algorithmic way) each of the regular factors in a minimal r-factorisation of a word 
	as a `Rees matrix triple' $(i,g,\lambda)$ where $g$ is an element of the maximal subgroup corresponding to the $\D$-class 
	$D$ of the factor in question, and $i,\lambda$ are indices labelling the $\R$- and the $\L$-classes of $D$, respectively. 
	The rewriting process in question is described in \textbf{Theorem \ref{thm:Rees}}.
	\item This sets up the introduction of the notion of \emph{contact automata} of two $\D$-classes in Sect.\ 5, which 
	eventually brings about the role of rational subsets of direct products of maximal subgroups of $\ig{\cE}$ mentioned previously. 
	We are then able to prove our main result, \textbf{Theorem \ref{prob-p}}, establishing the equivalence between the word problem 
	of $\ig{\cE}$ and a specific family of constraint satisfaction problems (CSPs) over (potentially infinite) finitely presented groups with rational predicates. 
\end{itemize}

Two principal applications of our general results are presented in Sect.\ 6 and 7, respectively. 
\begin{itemize}
  \item In Sect.\ 6, we prove \textbf{Theorem \ref{thm:appl}}, which states that if the maximal subgroups of $\ig{\cE}$ 
	arising from non-maximal $\D$-classes are finite (those arising from maximal ones must necessarily be free or trivial), 
	then the word problem of $\ig{\cE}$ is decidable. In particular, by the results of \cite{Do2,GR2}, this applies to 
	free idempotent generated semigroups over biorders of $\mathcal{T}_n$ and $\mathcal{PT}_n$, the monoids of all transformations 
	and all partial transformations on an $n$-element set, respectively (see \textbf{Corollary \ref{TnPTn}}). 
  \item On the other hand, in the final Sect.\ 7 we prove \textbf{Theorem \ref{thm:Fountain}}, showing that the statement 
	of one of the main results of \cite{YG}, Theorem 4.13, holds for $\ig{\cE}$ where $\cE$ is an arbitrary finite biordered set: 
	namely, in such a case, $\ig{\cE}$ is always \emph{weakly abundant}, along with satisfying the so-called congruence condition. 
	Weakly abundant semigroups \cite{ElQ,Lawson} (also coined \emph{Fountain semigroups} \cite{Ma,MS}) represent a natural generalisation 
	of regular semigroups, with a rich and interesting structural theory that is the subject of ongoing research. 
\end{itemize}

For the remainder of the article, we assume that $\cE$ is an arbitrary but fixed finite biordered set. We emphasise that
this is \emph{not} equivalent to `biordered set of a finite semigroup', since there are finite biordered sets (of infinite
semigroups) which do not arise from any finite semigroup, see Easdown \cite{Ea1}.


\section{Preliminaries}
\label{sec:prelim}

Throughout, we assume basic knowledge of semigroup theory and language theory. Still, in an attempt to keep the paper
reasonably self-contained, we gather here the most important prerequisites for our work.

\subsection*{The word problem}

We say that a semigroup $S$ is \emph{presented} by $\pre{A}{R}$, where $A$ is a set of letters and $R$ is a set of pairs
$(u_i,v_i)$, $i\in I$, of words over $A$ (usually written as formal equalities $u_i=v_i$), if $S$ is isomorphic to
$A^+/\rho_R$, the quotient of the free semigroup $A^+$ on $A$ by the congruence $\rho_R$ generated by $R$. There is
a natural homomorphism $\phi:A^+\to S$ whose kernel relation is precisely $\rho_R$. In particular, $A\phi$ is a generating
set of $S$, and hence the elements of $A$ are called the \emph{generators} and the pairs from $R$ the \emph{relations} of
the presentation. A presentation is \emph{finite} if both $A,R$ are finite sets. Note that the notion of a semigroup presentation
is fully analogous with that of a group presentation, with an appropriate free group taking the role of the free semigroup
on $A$, and bearing in mind that the notion of a group congruence is basically equivalent to that of a normal subgroup.

The immediate question that arises is: given two words $u,v\in A^+$, do they represent the same element of $S$, that is,
whether $u\phi=v\phi$ holds? This is the \emph{word problem} for $S=\pre{A}{R}$. More formally, we have the following
decision problem.

\bigskip

\noindent INPUT: $u,v\in A^+$.\\
OUTPUT: Decide if $u\phi=v\phi$, i.e.\ if $(u,v)\in\rho_R$.

\bigskip

Note that, in our specific case, if $\ig{\cE}$ would have been given by a presentation analogous to the above one without
bars, we would have $w\phi=\ol{w}$ for all $w\in E^+$. (However, we wanted, on one hand, to have the generators already
belonging to $\ig{\cE}$, and on the other to have a strong distinction between words over the generators and elements
of $\ig{\cE}$ represented by these words; hence the notational glitch already described in the introduction.) Therefore, the
word problem for $\ig{\cE}$ asks: given two words $u,v\in E^+$ decide if $\ol{u}=\ol{v}$ holds in $\ig{\cE}$. It is the main
goal of this paper to show how this problem boils down to a decision problem involving the maximal subgroups of $\ig{\cE}$
and certain rational subsets of their direct products, where both the presentation of these groups and the rational subsets
in question are computable from the biordered set $\cE$.

\subsection*{Basics of semigroup theory}

We use \cite{HoBook} as our basic reference in semigroup theory for all further notions and foundational results.

One of the most fundamental ideas in the course of studying the general structure of semigroups in to classify their elements
according to the (right, left, two-sided) principal ideals they generate. This is formalised via five equivalence relations
on a semigroup $S$ called \emph{Green's relations}. For $a,b\in S$ we define:
$$
a\;\R\; b \Leftrightarrow aS^1=bS^1, \quad a\;\L\; b \Leftrightarrow S^1a=S^1b,\quad a\;\J\; b \Leftrightarrow
S^1aS^1=S^1bS^1,
$$
where $S^1$ denotes $S$ with an identity element adjoined (unless $S$ already has one). We let $\H=\R\cap\L$ and let $\D=\R\vee\L$
be the least equivalence on $S$ containing both $\R$ and $\L$. It is well known (see e.g.\ \cite{HoBook}) that $\R\circ\L=\L\circ\R$ holds in
any semigroup, so $\D$ coincides with the latter relational composition. Also, $\R$ is always a left congruence, while $\L$
is a right congruence. In general we have $\H\subseteq \R,\L\subseteq\D\subseteq\J$. There is a large class of semigroups
(including all finite semigroups) for which $\D=\J$. For $\mathscr{K}\in\{\H,\R,\L,\D,\J\}$ we denote the $\mathscr{K}$-class
containing $a\in S$ by $K_a$ (where $K\in\{H,R,L,D,J\}$, respectively).

\begin{rmk}\label{rem:RLD}
Assume that the biordered set $\cE$ arises from the semigroup $S$. Then for any $e,f\in E=E(S)$ we have $e\;\R\;f$ in $\cE$ if and
only if $e\;\R\;f$ in $S$, and $e\;\L\;f$ in $\cE$ if and only if $e\;\L\;f$ in $S$. Indeed, if $e=fs$ and $f=et$ for some
$s,t\in S^1$ then clearly $fe=e$ and $ef=f$ (the converse implication is trivial), and a similar argument holds for the $\L$
relation. Hence, there is no harm at all in blurring the distinction between the relations $\R,\L$ in a birodered set and
their namesakes in any semigroup from which the biorder arises. By \cite[Lemma 2.5]{DGR} and the results of 
\cite{FG} (see Lemma \ref{lem:FG} below), if $S$ is idempotent
generated then the same remark holds for the relation $\D$ in $\cE$ and Green's relation $\D$ in $S$: two idempotents are
$\D$-related in $\cE$ if and only if they are $\D$-related in $S$. (We stress, however, that $\D=\R\circ\L=\L\circ\R$ does not
hold in general for biordered sets.)
\end{rmk}

The $\J$-classes of a semigroup can be partially ordered in a natural way upon defining $J_a\leq J_b$ if and only if $S^1aS^1
\subseteq S^1bS^1$. In a similar way one can define partial orders on collections of $\R$-classes and $\L$-classes of a semigroup.

\begin{rmk}
Since the relations $\D$ and $\J$ in general need not to coincide, it is possible that a certain $\J$-class can contain more
than one $\D$-class. Therefore, the $\J$-order $\leq$ defines a quasi-order on the set of $\D$-classes of a semigroup.
\end{rmk}

An element $a\in S$ is \emph{regular} if there exists $b\in S$ such that $aba=a$. We say that $a'\in S$ is an \emph{inverse}
of $a$ if $a=aa'a$ and $a'=a'aa'$. Clearly, any element having an inverse is necessarily regular, but the converse is true
as well: if $a=aba$ then $a'=bab$ is an inverse of $a$. It is well known that any $\D$-class of a semigroup consists either
entirely of regular or non-regular elements \cite[Proposition 2.3.1]{HoBook}, and in this sense we speak of regular and
non-regular $\D$-classes, respectively. In fact, regular $\D$-classes coincide with $\D$-classes containing idempotents.
Furthermore, for each idempotent $e$ its $\H$-class $H_e$ is a group with identity $e$; this is a maximal subgroup of $S$,
and all maximal subgroups of $S$ arise in this way. Any two maximal subgroups belonging to a single (regular) $\D$-class
must be isomorphic, so that the maximal subgroup is an invariant of a regular $\D$-class up to isomorphism.

We will repeatedly use the following result of Fitz-Gerald \cite{FG}.

\begin{lem}\label{lem:FG}
Let $S$ be an idempotent generated semigroup and let $a\in S$ be a regular element. Then there exist idempotents
$e_1,\dots,e_n\in D_a$ such that $a=e_1\dots e_n$.
\end{lem}

In other words, in idempotent generated semigroups, each element of a regular $\D$-class $D$ is a product of idempotents from $D$.

A particularly important class of semigroups are the ones that do not have any proper two-sided ideals. A semigroup $S$ is
called \emph{simple} if its only ideal is $S$ itself, and a semigroup with zero $0\in S$ is called \emph{0-simple} if $\{0\}$
and $S$ are its only ideals and $S^2\neq\{0\}$ (i.e.\ the multiplication in $S$ is not null). A 0-simple semigroup is
\emph{completely 0-simple} if it has a \emph{primitive idempotent} -- an idempotent which is minimal (with respect to the
natural order $\le$) among the non-zero ones.

A foundational result called the \emph{Rees Theorem} provides a particularly nice representation of completely
0-simple semigroups. It states that $S$ is completely 0-simple if and only if it is isomorphic to a \emph{Rees matrix semigroup}
$\mathcal{M}^0[G;I,\Lambda;P]$, where $G$ is a group, $I,\Lambda$ are index sets and $P$ is a regular $\Lambda\times I$ matrix
over $G\cup\{0\}$, which means that each row and column contains a non-zero entry. Elements of such a Rees matrix semigroup are
$0$ and triples of the form $(i,g,\lambda)$ where $g\in G$, $i\in I$, $\lambda\in\Lambda$; the multiplication is defined by
$00=0(i,g,\lambda)=(i,g,\lambda)0=0$ and
$$
(i,g,\lambda)(j,h,\mu) = \left\{\begin{array}{ll}
(i,gp_{\lambda j}h,\mu) & \text{if }p_{\lambda j}\neq 0,\\
0 & \text{otherwise.}
\end{array}\right.
$$

One canonical way in which completely 0-simple semigroups might arise are \emph{principal factors} of a semigroup. Let $J$ be
a $\J$-class. The associated principal factor $\ol{J}=J\cup\{0\}$ is defined by the multiplication
$$
x\cdot y = \left\{\begin{array}{ll}
xy & \text{if }x,y,xy\in J,\\
0 & \text{otherwise.}
\end{array}\right.
$$
Any principal factor is either 0-simple or has a zero multiplication; if $J$ contains an idempotent the former case takes place.
Now, a regular $\D$-class of a semigroup may or may not coincide with its $\J$-class. However, in this paper we are exclusively
dealing with semigroups with finitely many idempotents (that is, finite biordered sets), in which case the 
situation simplifies significantly.

\begin{lem}\label{lem:DJ}
Let $S$ be a semigroup with finitely many idempotents, and let $D$ be a regular $\D$-class of $S$. Then $D$ coincides with its
$\J$-class, and the corresponding principal factor is completely 0-simple.
\end{lem}

\begin{proof}
Let $J$ be the $\J$-class containing $D$; then $\ol{J}$ is a 0-simple semigroup. Furthermore, $\ol{J}$ has finitely many (non-zero)
idempotents, so it contains a primitive one, implying that $\ol{J}$ is completely 0-simple. So, by \cite[Lemma 3.2.7]{HoBook} all
elements of $J$ are $\D$-related in $\ol{J}$ and thus in $S$, showing that $D=J$.
\end{proof}

Hence, since in all our considerations $\ig{\cE}$ has finitely many idempotents, the quasi-order induced on $\D$-classes by the
$\J$-order turns into a partial order when restricted to the collection of \emph{regular} $\D$-classes (and coincides with
the $\J$-order on these classes). For these reasons, if $D_1,D_2$ are two regular $\D$-classes of $\ig{\cE}$ it will make
sense to write $D_1\leq D_2$ with respect to the $\J$-order $\leq$.

\subsection*{Properties of $\ig{\cE}$}

It was proved in \cite{Ea2} that the biordered set of idempotents of $\ig{\cE}$ consists precisely of
$\{\ol{e}:\ e\in E\}$ and is thus isomorphic to $\cE$. Furthermore, it was noted in \cite{GR1} that if
$\cE$ arises from an idempotent generated semigroup $S$ then there is a (surjective) homomorphism $\varphi:
\ig{\cE}\to S$ which, as a consequence of \cite{FG}, maps the $\R$-class ($\L$-class, respectively) of $\ol{e}$
in $\ig{\cE}$ onto the $\R$-class ($\L$-class, resp.) of $e$ in $S$. Therefore, given $D=D_{\ol{e}}$, a
(regular) $\D$-class in $\ig{\cE}$, there is a bijection between the set of $\R$-classes (resp.\ $\L$-classes)
contained in $D$ and the corresponding set in the $\D$-class of $e$ in $S$. Combined with the facts discussed 
in Remark \ref{rem:RLD}, this means that, once we are given a (finite) biordered set $\cE$, we already know some
substantial information about the regular $\D$-classes of $\ig{\cE}$. Namely, if $D_1,\dots,D_m$ is the list
of all $\D$-classes of $\cE$, then $\ig{\cE}$ will also have precisely $m$ regular $\D$-classes, say $D'_1,\dots,
D'_m$, and the `shapes' of corresponding $\D$-classes will be the same: for $1\leq k\leq m$, if $I_k$ is an index
set enumerating $D_k/\R$ and $\Lambda_k$ enumerates $D_k/\L$, then the number of $\R$-($\L$-)classes in $D'_k$
will be $|I_k|$ (resp.\ $|\Lambda_k|$). For this reason, there is no harm in slightly abusing notation and assuming
that $I_k$ and $\Lambda_k$ also enumerate the $\R$-classes (resp.\ $\L$-classes) of $D'_k$. In addition,
by Lemma \ref{lem:DJ} we know that each $D'_k$ will coincide with its $\J$-class and so the $\J$-order
imposes a partial order on the regular $\D$-classes of $\ig{\cE}$.

One of the main results of \cite{DGR} is that regular elements of $\ig{\cE}$ can be effectively
recognised, and furthermore, that the word problem of $\ig{\cE}$ is decidable in its regular part provided
the maximal subgroups of $\ig{\cE}$ all have decidable word problems (in the sense that if $u,v\in E^+$,
there is an algorithm which decides if $\ol{u},\ol{v}$ are regular in $\ig{\cE}$, and if they are, decides
if $\ol{u}=\ol{v}$). Here we need just the decidability of regularity in $\ig{\cE}$, and
we summarise the relevant statements (Lemma 3.3(ii), Theorems 3.6 and 3.7 of \cite{DGR}) in the following.
Here, and in the remainder of the paper, for an alphabet $A$, $A^\ast$ denotes the free monoid on $A$, which is
just $A^+$ augmented with the empty word.

\begin{thm}\label{thm:DGR}
\begin{itemize}
\item[(i)] There exists an algorithm which, given $e\in E$ and $v\in E^\ast$, decides whether
$\ol{ev}\;\R\;\ol{e}$ holds in $\ig{\cE}$ and if so, returns an $f\in E$ such that $\ol{ev}\;\L\;\ol{f}$.
Dually, there exists an algorithm which, given $e\in E$ and $u\in E^\ast$, decides whether
$\ol{ue}\;\L\;\ol{e}$ holds in $\ig{\cE}$ and if so, returns a $g\in E$ such that $\ol{ue}\;\R\;\ol{g}$.
\item [(ii)] For any $w\in E^+$, $\ol{w}$ is a regular element of $\ig{\cE}$ if and only if $w$ can be
factorised as
$$
w = uev
$$
such that $e\in E$ and $\ol{ue}\;\L\;\ol{e}\;\R\;\ol{ev}$. In such a case, $\ol{e}\;\D\;\ol{w}$.
\item[(iii)] There exists an algorithm which, given $w\in E^+$, decides whether $\ol{w}$ is a regular
element of $\ig{\cE}$, and if so, returns $e,f\in E$ such that $\ol{e}\;\R\;\ol{w}\;\L\;\ol{f}$.
\end{itemize}
\end{thm}

\begin{rmk}\label{the-rem}
\begin{itemize}
\item[(i)]
Note that if $\ol{e}\;\R\;\ol{ev}$ holds in $\ig{\cE}$ then for any prefix $v'$ of $v$ we have $\ol{e}\;\R\;\ol{ev'}$
and thus, since $\R$ is a left congruence, $\ol{ze}\;\R\;\ol{zev'}$ for any word $z\in E^\ast$. Analogously,
$\ol{ue}\;\L\;\ol{e}$ implies $\ol{u'ez}\;\L\;\ol{ez}$ for any suffix $u'$ of $u$ and any word $z\in E^\ast$.
In particular, for a factorisation $w=uev$ as in part (ii) of the previous theorem we have 
$\ol{ue}\;\R\;\ol{w}\;\L\;\ol{ev}$.
\item[(ii)]
The following converse of a part of the statement (ii) from the previous theorem also holds: if $w\in E^+$
has a factorisation of the form $w=uev$ such that $\ol{e}\;\D\;\ol{w}$ then $\ol{w}$ represents
a regular element of $\ig{\cE}$ and we have $\ol{ue}\;\L\;\ol{e}\;\R\;\ol{ev}$. This follows
from Lemma \ref{lem:DJ}, the inequalities $J_{\ol{e}}=J_{\ol{uev}}\leq J_{\ol{ue}},J_{\ol{ev}}\leq J_{\ol{e}}$
(implying that $\ol{e},\ol{ue},\ol{ev},\ol{w}$ all belong to the same $\D$-class with a completely 0-simple
principal factor), and the basic properties of completely 0-simple semigroups (see \cite{HoBook}).
\end{itemize}
We will freely use these facts in the next section without further reference.
\end{rmk}

\subsection*{Rational subsets of monoids}

Let $M$ be a monoid. The set $\Rat(M)$ of \emph{rational subsets} of $M$ is the smallest subset of the power
set $\mathcal{P}(M)$ which contains all the finite subsets of $M$ and is closed for the operations of taking
\begin{itemize}
\item[(1)] the union of two sets;
\item[(2)] the product of two sets, $A\cdot B=\{ab:\ a\in A, b\in B\}$;
\item[(3)] the \emph{Kleene star} $A^\ast$ of a set $A\subseteq M$, which is just the submonoid of $M$
generated by $A$.
\end{itemize}

A classical result in language theory (Kleene's Theorem) \cite{HUBook} states that $L\subseteq\Sigma^\ast$ is rational if and only
if it is recognised by a finite state automaton. It is easy to see that if $M$ happens to be generated by a finite set $\Sigma$ and $\psi:
\Sigma^\ast\to M$ is the canonical monoid homomorphism then $A$ is a rational subset of $M$ if and only if
$A=L\psi$ for some rational language $L\subseteq\Sigma^\ast$. If the latter is indeed the case then we have
two options of effectively specifying a rational subset of $M$: we can define it either via a finite automaton
over $\Sigma$, or by means of a rational expression over $\Sigma$. The two approaches are equivalent by virtue
of Kleene's Theorem because its proof provides effective algorithms which convert a finite automaton into
the rational expression representing the language that the automaton accepts, and vice versa. Let us conclude
by noting that if $M$ is a group, then the previous remarks hold provided $\Sigma$ is a \emph{monoid} generating
set of $M$. If, however, $\Gamma$ generates $M$ as a group, then the symmetric generating set $\Gamma^{\pm 1}=
\Gamma\cup\Gamma^{-1}$ takes the role of $\Sigma$.

\section{Minimal r-factorisations and the $\D$-fingerprint}
\label{sec:min-r}

Let $w\in E^+$. A factorisation
$$
w = w_1\dots w_m
$$
is called an \emph{r-factorisation} of $w$ (with respect to $\cE$)
if $\ol{w_i}$ is a regular element of $\ig{\cE}$ for all $1\leq i\leq m$.
Note that there is a natural partial order on the set of all factorisations of $w$: if $w=p_1\dots p_m=q_1\dots q_s$
we say that $(p_1,\dots,p_m)$ is \emph{coarser} than $(q_1,\dots,q_s)$ and write
$$(p_1,\dots,p_m)\preceq (q_1,\dots,q_s)$$
if $m\leq s$ and there exist $1=s_1<s_2<\dots<s_m\leq s$ with $$p_i=q_{s_i}\dots q_{s_{i+1}-1}$$
for all $1\leq i<m$ and $p_m=q_{s_m}\dots q_s$. This order has a maximum element -- the factorisation of $w$ into
its letters (which is at the same time an r-factorisation, so at least one r-factorisation exists for any word $w$)
-- and a minimum element, the trivial factorisation. This order can be also restricted to r-factorisations
only, thus giving rise to the notion of a \emph{minimal r-factorisation} of $w$ (somewhat akin to the \emph{almost
normal forms} of \cite{YG}). Note that the set of (r-)factorisations
of a word is finite and non-empty, so any word has at least one minimal r-factorisation; it is easy to see that an
r-factorisation $w = w_1\dots w_m$ is minimal if and only if for any $1\leq i\leq j\leq m$ the assumption that
$\ol{w_i\dots w_j}$ is regular implies that $i=j$: no product of at least two consecutive factors represents
a regular element of $\ig{\cE}$.

Certainly, a minimal r-factorisation of a word need not to be unique. Nevertheless, we are going to prove not only that all
r-factorisations of a single word $w$ `look alike', but that when we are concerned with two words $u,v\in E^+$
such that $\ol{u}=\ol{v}$ their arbitrary minimal r-factorisations behave like a sort of `pseudo-normal form':
they have the same number of factors, and the two (regular) elements of $\ig{\cE}$ represented by any pair of
corresponding factors belong to the same $\D$-class. This will naturally give rise to the notion of a
\emph{$\D$-fingerprint} of an element of $\ig{\cE}$, a sequence of regular $\D$-classes of $\ig{\cE}$ uniquely
determined by any minimal r-factorisation of an arbitrary word $w\in E^+$ representing the considered
element. In fact, it will be shown that the entire semigroup $\ig{\cE}$ can be identified with a quotient of the set
of all finite sequences of regular elements of $\ig{\cE}$ in which no product of two or more consecutive
elements is regular; thus the word problem of $\ig{\cE}$ will boil down to the issue of the computability of
a certain equivalence on that set (where sharing the same length and $\D$-fingerprint is a necessary condition
for equivalence). Ultimately, we will show that this computability issue reduces to the question of satisfying
a certain existential formula involving rational subsets of direct products of pairs of maximal subgroups of
$\ig{\cE}$, a purely group-theoretical algorithmic problem.

To this end, we define a relation $\approx$ on the set of all (non-empty) finite sequences of words from $E^+$
by setting
$$
(p_1,\dots,p_m) \approx (q_1,\dots,q_s)
$$
if and only if $m=s$ and one of the three following conditions hold:
\begin{enumerate}
\item[(i)] $\ol{p_i}=\ol{q_i}$ for some $1\leq i\leq m$ and $p_j=q_j$ for all $j\neq i$;
\item[(ii)] $\ol{p_i}=\ol{q_ie}$ and $\ol{q_{i+1}}=\ol{ep_{i+1}}$ for some $1\leq i<m$ and $e\in E$,
and $p_j=q_j$ for all $j\not\in\{i,i+1\}$;
\item[(iii)] $\ol{q_i}=\ol{p_ie}$ and $\ol{p_{i+1}}=\ol{eq_{i+1}}$ for some $1\leq i<m$ and $e\in E$,
and $p_j=q_j$ for all $j\not\in\{i,i+1\}$.
\end{enumerate}
This relation is clearly reflexive and symmetric. We define the relation $\sim$ to be the transitive closure
of $\approx$. Note that if $(p_1,\dots,p_m)$ and $(q_1,\dots,q_m)$ are such that $\ol{p_i}=
\ol{q_i}$ for all $1\leq i\leq m$ then $(p_1,\dots,p_m)\sim (q_1,\dots,q_m)$. Also, the following is immediate
from the above definition.

\begin{lem}\label{lem0}
If $(p_1,\dots,p_m)\sim (q_1,\dots,q_m)$ then $\ol{p_1\dots p_m}=\ol{q_1\dots q_m}$.
\end{lem}

Let $w\in E^+$ be a word representing a regular element $\ol{w}$ of $\ig{\cE}$. By Theorem \ref{thm:DGR}(ii),
there is a factorisation $w=uev$ satisfying $\ol{ue}\;\L\;\ol{e}\;\R\;\ol{ev}$. In such a case, the letter
(idempotent) $e$ will be called a \emph{seed} of $w$. Clearly, seeds need not be unique in words representing
regular elements. (In fact, by Lemma \ref{lem:FG} there exist idempotents $\ol{e_1},\dots,\ol{e_n}
\in D_{\ol{w}}=D_{\ol{e}}$ such that $\ol{e_1\dots e_n}=\ol{w}$. Therefore,
it can happen that every single letter of a word is a seed.)

We write $w'\equiv w$ if the word $w'$ is obtained from $w$ by application of \emph{one} rewriting rule
induced by the defining relations of the presentation of $\ig{\cE}$ -- that is to say, either a two-letter
subword $ef$ is replaced by the letter $e\ast f$, provided $\{e,f\}$ is a basic pair, or a single letter $g$ is
replaced by a two-letter word $ef$ such that $\{e,f\}$ is a basic pair and $e\ast f=g$. Note that $\equiv$
is a symmetric relation.

\begin{lem}\label{lem1}
Let $w\in E^+$ and $w=p_1\dots p_m$ be a minimal r-factorisation of $w$. Furthermore, let $w'$ be a word such
that $w'\equiv w$. Then $w'$ has a minimal r-factorisation $w'=p_1'\dots p_m'$ such that
$(p_1,\dots,p_m)\approx (p_1',\dots,p_m')$. Furthermore, we have either $\ol{p_i}\;\R\;\ol{p_i'}$ or
$\ol{p_i}\;\L\;\ol{p_i'}$ (and thus $\ol{p_i}\;\D\;\ol{p_i'}$) for all $1\leq i\leq m$.
In particular, we have $\ol{p_1}\;\R\;\ol{p_1'}$ and $\ol{p_m}\;\L\;\ol{p_m'}$.
\end{lem}

\begin{proof}
If $w'$ is obtained from $w$ such that a two-letter subword $ef$ of a factor $p_i$ (for some $1\leq i\leq m$)
is replaced by $e\ast f$, or, the other way round, that a letter $g$ occurring in $p_i$ is replaced by the
word $ef$ with the properties as indicated above -- thus turning $p_i$ into $p_i'\equiv p_i$ -- then we clearly
have $\ol{p_i'}=\ol{p_i}$ and so condition (i) takes place in the definition of $\approx$ (as all other factors 
of $w$ remain unchanged), which immediately yields the lemma.

Hence, it remains to discuss the case when a subword $ef$ of $w$ is replaced by $e\ast f$ in such a way that
$e$ is the last letter of $p_i$ and $f$ is the first letter of $p_{i+1}$ for some $1\leq i<m$. Throughout the
remainder of the proof we write $p_i=qe$ and $p_{i+1}=fp'$. There are three subcases to consider: $e\ast f=e$, 
$e\ast f=f$, and $e\ast f\not\in\{e,f\}$. We look only at the first and the third one, as the second one is 
analogous to the first.

So, let $e\ast f=e$. We argue that $\ol{p'}\;\L\;\ol{p_{i+1}}$; indeed, this will follow as soon as we show 
that $f$ cannot be a seed for $p_{i+1}$ (because then it follows quickly from Theorem \ref{thm:DGR}(ii)
that $\ol{p'}\;\L\;\ol{p_{i+1}}$ is regular and a seed for $p_{i+1}$ is also a seed for $p'$). 
Assuming to the contrary, it follows that $\ol{f}\;\R\;\ol{p_{i+1}}=\ol{fp'}$.
Upon multiplying by $\ol{qe}$ from the left, the fact that $\R$ is a left congruence
implies that $\ol{qef}=\ol{q}\,\ol{ef}=\ol{q}\,\ol{e\ast f}=\ol{qe}=\ol{p_i}$ is $\R$-related to
$\ol{p_ip_{i+1}}$, forcing the latter to be regular, which is a contradiction to the minimality assumption.
Therefore, if we let $p_i'=p_if=qef$ and $p_{i+1}'=p'$ and leave all the other factors unchanged, then
$\ol{p_i}=\ol{p_i'}$ and we have condition (iii) in the definition of $\approx$, so the lemma holds in this case.

Finally, let $e\ast f=g\not\in\{e,f\}$; this happens because $f\ast e\in\{e,f\}$. Assume that
$f\ast e=f$ (the other case is analogous). Then $\ol{g}\;\L\;\ol{f}$, as $\ol{fg}=\ol{fef}=\ol{f}$, implying
(since $\L$ is a right congruence) $\ol{gp'}\;\L\;\ol{fp'}=\ol{p_{i+1}}$. Similarly to the previous paragraph, 
we must have $\ol{q}\;\R\;\ol{p_i}$ unless $e$ is
a seed for $p_i$. To show the latter is impossible, assume $\ol{e}\;\L\;\ol{p_i}=\ol{qe}$. By multiplying by
$\ol{p_{i+1}}=\ol{fp'}$ from the right, we get
$$
\ol{p_{i+1}}\;\L\;\ol{gp'}=\ol{efp'}\;\L\;\ol{p_ip_{i+1}},
$$
so that $\ol{p_ip_{i+1}}$ is regular, contradicting, again, the minimality of the initial factorisation.
Hence, with $p_i'=q$ and $p_{i+1}'=gp'$ (and other factors unchanged) we have $\ol{p_i}\;\R\;\ol{p_i'}$
and $\ol{p_{i+1}}\;\L\;\ol{p_{i+1}'}$ and condition (ii) from the definition of $\approx$. This completes
the proof of the lemma, as $\ol{p_1}\;\R\;\ol{p_1'}$ and $\ol{p_m}\;\L\;\ol{p_m'}$ follow immediately from 
previous considerations.
\end{proof}

Next we want to show that any two minimal r-factorisations of a single word are `similar'. For this, we introduce
the concept of the position of a letter within a word. Namely, if $w=uev$ defines an occurrence of the letter $e$
within $w$ we say that the \emph{position} of this occurrence is the integer $|u|+1=|ue|$. Similarly, if we have
a factorisation $w=w_1\dots w_m$ then we can `coordinatise' it by recording the
sequence $(\alpha_1,\dots,\alpha_m)$ of positions of first letters (from the left) of the factors $w_1,\dots,w_m$.
Here necessarily $\alpha_1=1$ and $1<\alpha_2<\dots<\alpha_m\leq |w|$.

\begin{lem}\label{lem2}
Let
$$
w = p_1p_2\dots p_m = q_1q_2\dots q_s
$$
be two minimal r-factorisations of a word $w\in E^+$ with coordinates $(1,\alpha_2,\dots,\alpha_m)$ and
$(1,\beta_2,\dots,\beta_s)$, respectively. Assume that for each $p_i$, $1\leq i\leq m$, and $q_j$, $1\leq j\leq s$,
we have fixed one seed -- say, $e_i$ at position $\gamma_i$ and $f_j$ at position $\delta_j$, respectively --
so that we have
$$
1\leq\gamma_1<\alpha_2\leq\gamma_2<\dots<\alpha_m\leq\gamma_m\leq|w|
$$
and
$$
1\leq\delta_1<\beta_2\leq\delta_2<\dots<\beta_s\leq\delta_s\leq|w| .
$$
Then $m=s$, $\ol{p_i}\;\D\;\ol{q_i}$ for all $1\leq i\leq m$, and
$$
1\leq\gamma_1,\delta_1<\alpha_2,\beta_2\leq\gamma_2,\delta_2<\dots<\alpha_m,\beta_m\leq\gamma_m,\delta_m\leq|w| .
$$
In addition, $\ol{p_1}\;\R\;\ol{q_1}$ and $\ol{p_m}\;\L\;\ol{q_m}$.
\end{lem}

\begin{proof}
First of all, if $m=1$ then $\ol{w}=\ol{p_1}$ is regular, so the minimality of the r-factorisation 
$w=q_1q_2\dots q_s$ implies $s=1$. Similarly, $s=1$ implies $m=1$, and the lemma follows immediately. Hence,
for the rest of the proof we may assume that $m,s\geq 2$. We start by discussing the possible relationships 
between the words $p_1$ and $q_1$ and their seeds, thus establishing the inequalities 
$1\leq\gamma_1,\delta_1<\alpha_2,\beta_2\leq\gamma_2,\delta_2$.

Consider first the possibility that $\beta_2\leq \gamma_1$, i.e.\ that the word $q_1$ ends before reaching the
distinguished seed $e_1$ of $p_1$. Let $k\geq 2$ be the unique index with $\beta_k\leq\gamma_1<\beta_{k+1}$
(throughout the proof, we let $\alpha_{m+1}=\beta_{s+1}=|w|+1$). Now we can write $p_1=u_1e_1v_1$ and
$q_k=q_k'e_1q_k''$, so that $u_1=q_1\dots q_{k-1}q_k'$. By Theorem \ref{thm:DGR}(ii) we have that $\ol{e_1}\;\L\;
\ol{q_k'e_1}\;\L\;\ol{u_1e_1}$. After multiplication by $\ol{q_k''}$ from the right, this becomes
$$\ol{q_1\dots q_k}=\ol{u_1e_1q_k''}\;\L\;\ol{q_k'e_1q_k''}=\ol{q_k},$$
implying that $\ol{q_1\dots q_k}$ is regular; this contradicts the minimality of the second factorisation of $w$.
Therefore, we must have $\gamma_1<\beta_2$, and by switching the roles of the two factorisations we analogously
conclude that $\delta_1<\alpha_2$.

Now suppose that $\gamma_2<\beta_2$. Then
$$
q_1 = p_1\dots p_{k-1}p_k'
$$
for some $k\geq 2$ as large as possible such that $p_k'$ is a prefix of $p_k$. From the previous paragraph,
the distinguished seed $f_1$ of $q_1$ appears in $p_1$, so we have a factorisation $q_1=z_1f_1w_1$ (in the
sense of Theorem \ref{thm:DGR}(ii)), and also $p_1=z_1f_1w_1'$, for some words $z_1,w_1,w_1'$ such that 
$w_1'$ is a prefix of $w_1$. Hence,
$$
\ol{q_1}\;\R\;\ol{z_1f_1}\;\R\;\ol{p_1}\;\R\;\ol{p_1\dots p_{k-1}}.
$$
If $k-1\geq 2$ then we have a contradiction, by minimality of the r-factorisation: thus $k=2$. However, the
assumption $\gamma_2<\beta_2$ implies that $p_2'$ must contain the distinguished seed $e_2$ of $p_2$. We
therefore have factorisations $p_2=u_2e_2v_2$ (in the sense of Theorem \ref{thm:DGR}(ii)) and $p_2'=u_2e_2v_2'$,
for some words $u_2,v_2,v_2'$ such that $v_2'$ is a prefix of $v_2$. Hence
$$
\ol{q_1} = \ol{p_1p_2'} = \ol{p_1u_2e_2v_2'}\;\R\;\ol{p_1u_2e_2}\;\R\;\ol{p_1p_2},
$$
again contradicting minimality of the given r-factorisations. Thus we must conclude $\beta_2\leq\gamma_2$, 
and by switching the roles of $p$'s and $q$'s we get $\alpha_2\leq\delta_2$.

So, we now know that we must have $1\leq\gamma_1,\delta_1<\alpha_2,\beta_2\leq\gamma_2,\delta_2$. Assume, by
induction, that we have already established the inequalities $\gamma_{k-1},\delta_{k-1}<\alpha_k,\beta_k\leq
\gamma_k,\delta_k$ for some $k\geq 2$.

First we want to prove that $\gamma_k<\beta_{k+1}$, so assume to the contrary that $\beta_{k+1}\leq\gamma_k$.
Let $h\geq k+1$ be the unique index with $\beta_h\leq\gamma_k<\beta_{h+1}$, so that the distinguished seed $e_k$ of
$p_k$ is positioned within $q_h$. We write $q_h=ye_kz$. If $\alpha_k<\beta_k$ then $q_k\dots q_{h-1}ye_k$ is
a suffix of $u_ke_k$, where $p_k=u_ke_kv_k$ in the sense of Theorem \ref{thm:DGR}(ii), so that 
$$
\ol{q_k\dots q_{h-1}ye_k} \;\L\; \ol{e_k} \;\L\; \ol{ye_k},
$$
whence, multiplying through on the right by $\ol{z}$, we deduce $\ol{q_k\dots q_h}\;\L\;\ol{q_h}$, contradicting
minimality. Therefore, we may continue working under the assumption that $\alpha_k\geq\beta_k$. So, if we
write $q_k=u_k'f_kv_k'$ in the sense of Theorem \ref{thm:DGR}(ii), and bearing in mind that we have 
$\alpha_k\leq\delta_k$ as granted, we
can factorise $u_k'=xx'$, where $x'$ is the prefix of $p_k$ between positions $\alpha_k$ and $\delta_k-1$ (empty
if $\alpha_k=\delta_k$). We now have
$$\ol{ye_k}\;\L\;\ol{e_k}\;\L\;\ol{x'f_kv_k'q_{k+1}\dots q_{h-1}ye_k}$$
(note that $x'f_kv_k'q_{k+1}\dots q_{h-1}y$ coincides with $u_k$, the prefix preceding $e_k$ in $p_k$),
which by multiplying by $z$ from the right gives
$$\ol{q_h}=\ol{ye_kz}\;\L\;\ol{x'f_kv_k'q_{k+1}\dots q_{h-1}ye_kz}=\ol{x'f_kv_k'q_{k+1}\dots q_{h-1}q_h}.$$
On the other hand, $\ol{u_k'f_k}\;\L\;\ol{f_k}\;\L\;\ol{x'f_k}$ implies
$$\ol{q_k\dots q_h}=\ol{u_k'f_kv_k'q_{k+1}\dots q_{h-1}q_h}\;\L\;\ol{x'f_kv_k'q_{k+1}\dots q_{h-1}q_h}.$$
We conclude that $\ol{q_k\dots q_h}$ is $\L$-related to the regular element $\ol{q_h}$, a contradiction.
So, $\gamma_k<\beta_{k+1}$, and for analogous reasons $\delta_k<\alpha_{k+1}$.

Now we want to show that $\beta_{k+1}\leq\gamma_{k+1}$. To this end, assume that $\gamma_{k+1}<\beta_{k+1}$.
Putting this together with the induction hypothesis, we get $\beta_k\leq \gamma_k<\gamma_{k+1}<\beta_{k+1}$.
This means that $q_k$ contains the positions of both $e_k$ and $e_{k+1}$, the distinguished seeds of
$p_k$ and $p_{k+1}$, respectively. Also, we already know that $\alpha_k\leq\delta_k<\alpha_{k+1}$, placing
$f_k$, the distinguished seed of $q_k$, within $p_k$. Let $l$ be the least number with the property $\alpha_l
\geq \beta_{k+1}$. By the information we already gathered, $l\geq k+2$; we distinguish
between the cases $l=k+2$ and $l\geq k+3$. In either case, write $p_k=u_ke_kv_k$, $p_{k+1}=u_{k+1}e_{k+1}v_{k+1}$
and $q_k=u_k'f_kv_k'$ in the sense of Theorem \ref{thm:DGR}(ii).

Let first $l=k+2$. Let $x$ be the prefix of $v_{k+1}$ ending at position $\beta_{k+1}-1$. If
$\alpha_k\leq\beta_k$ then let $y$ be the suffix of $u_k$ starting at the position $\beta_k$.
Then $\ol{u_ke_k}\;\L\;\ol{ye_k}$, yielding
$$\ol{p_ku_{k+1}e_{k+1}x}=\ol{u_ke_kv_ku_{k+1}e_{k+1}x}\;\L\;\ol{ye_kv_ku_{k+1}e_{k+1}x}=\ol{q_k}.$$
Otherwise, if $\beta_k<\alpha_k$, let $z$ be the suffix of $u_k'$ beginning at position $\alpha_k$. So,
$\ol{u_k'f_k}\;\L\;\ol{zf_k}$ implying
$$\ol{q_k}=\ol{u_k'f_kv_k'}\;\L\; \ol{zf_kv_k'} = \ol{p_ku_{k+1}e_{k+1}x}.$$
Thus in any case $\ol{q_k}\;\L\;\ol{p_ku_{k+1}e_{k+1}x}$. However, we also have $\ol{e_{k+1}x}\;\R\;
\ol{e_{k+1}v_{k+1}}$, which implies $$\ol{p_ku_{k+1}e_{k+1}x}\;\R\;\ol{p_ku_{k+1}e_{k+1}v_{k+1}}=\ol{p_kp_{k+1}}.$$
Hence, we obtain $\ol{q_k}\;\D\;\ol{p_kp_{k+1}}$, implying the latter to be a regular element, a contradiction.

Now turn to the case $l\geq k+3$. Let $x$ now be a prefix of $v_k'$ ending at position $\alpha_{l-1}-1$, so that
$\ol{f_kv_k'}\;\R\;\ol{f_kx}$ and $\ol{q_k}=\ol{u_k'f_kv_k'}\;\R\;\ol{u_k'f_kx}$. If $\beta_k\leq\alpha_k$,
then let $y$ be the suffix of $u_k'$ starting at position $\alpha_k$. Then $\ol{u_k'f_k}\;\L\;\ol{yf_k}$ and
$$\ol{u_k'f_kx}\;\L\;\ol{yf_kx}=\ol{p_k\dots p_{l-2}},$$
so $\ol{q_k}\;\D\;\ol{p_k\dots p_{l-2}}$, prompting the
latter product to be regular, a contradiction. If, however, $\alpha_k<\beta_k$, then let $z$ be the suffix of $u_k$
starting at position $\beta_k$. Accordingly, $\ol{u_ke_k}\;\L\;\ol{ze_k}$ and
$$\ol{p_k\dots p_{l-2}}=
\ol{u_ke_kv_kp_{k+1}\dots p_{l-2}}\;\L\;\ol{ze_kv_kp_{k+1}\dots p_{l-2}} = \ol{u_k'f_kx},$$
yielding the same impossible conclusion, $\ol{q_k}\;\D\;\ol{p_k\dots p_{l-2}}$.

Hence, indeed we must have $\beta_{k+1}\leq\gamma_{k+1}$ and, dually, $\alpha_{k+1}\leq\delta_{k+1}$, thus
completing the inductive proof that the required inequalities $\gamma_{k-1},\delta_{k-1}<\alpha_k,\beta_k\leq
\gamma_k,\delta_k$ hold for $k\leq\min(m,s)$.

With the aim of showing that $m=s$, assume without loss of generality that $m<s$. Then we have $\alpha_m,\beta_m\leq
\gamma_m,\delta_m<\beta_{m+1}\leq\delta_{m+1}<\dots<\beta_s\leq \delta_s$, and so $p_m$ contains $q_{m+1}\dots q_s$
as a suffix; in addition, it contains the distinguished seed $f_m$ of $q_m$. Conversely, the seed $e_m$ of $p_m$
is contained in $q_m$, so we may write $p_m=u_me_mv_m$ (in the sense of Theorem \ref{thm:DGR}(ii)) and $q_m=xe_my$
(note that the latter factorisation does not imply that $e_m$ is a seed for $q_m$). Therefore, $\ol{e_mv_m}\;\R\;\ol{e_my}$
and so $\ol{q_m}=\ol{xe_my}\;\R\;\ol{xe_mv_m}=\ol{q_m\dots q_s}$, since $v_m=yq_{m+1}\dots q_s$. It follows
that $\ol{q_m\dots q_s}$ is a regular element, a contradiction. This proves $m=s$.

Finally, consider the words $p_i$ and $q_i$ and recall that now we already know that $\alpha_i,\beta_i\leq
\gamma_i,\delta_i < \alpha_{i+1},\beta_{i+1}$. Without any loss of generality, assume that $\alpha_i\leq\beta_i$.
Write $p_i=u_ie_iv_i$ and $q_i=u_i'f_iv_i'$ (with respect to their distinguished seeds), and let $x$ be the
suffix of $u_i$ beginning at position $\beta_i$. Then $\ol{u_ie_i}\;\L\;\ol{xe_i}$. If $\alpha_{i+1}\leq\beta_{i+1}$,
multiply the latter relation by $v_i$ from the right to get $\ol{p_i}\;\L\;\ol{xe_iv_i}$. In particular, if $i=m$,
we have $\ol{p_m}\;\L\;\ol{q_m}$ since $q_m=xe_iv_i$ (this is the suffix of $w$ starting at position $\beta_m$). 
On the other hand, let $y$ be a prefix of $v_i'$ ending at position $\alpha_{i+1}-1$; then $\ol{f_iv_i'}\;\R\;\ol{f_iy}$ 
and $\ol{q_i}=\ol{u_i'f_iv_i'}\;\R\;\ol{u_i'f_iy}$. Since $xe_iv_i=u_i'f_iy$ (both sides are equal to the subword
of $w$ between positions $\beta_i$ and $\alpha_{i+1}-1$), we have $\ol{p_i}\;\D\;\ol{q_i}$, as required. For 
$i=1$ we obtain $\ol{p_1}\;\R\;\ol{q_1}$, as $p_1=u_i'f_iy$ is the prefix of $q_1$ ending at position $\alpha_2-1$.
A similar reasoning leads to the same conclusion if $\alpha_{i+1}>\beta_{i+1}$, whence we are done.
\end{proof}

\begin{thm}\label{D-fing}
Let $u,v\in E^+$ be such that $\ol{u}=\ol{v}$. Furthermore, let $u=p_1\dots p_m$ and $v=q_1\dots q_s$ be arbitrary
minimal r-factorisations. Then $m=s$ and for all $1\leq i\leq m$ we have $\ol{p_i}\;\D\;\ol{q_i}$. Furthermore, 
we have $\ol{p_1}\;\R\;\ol{q_1}$ and $\ol{p_m}\;\L\;\ol{q_m}$.
\end{thm}

\begin{proof}
First of all, we have $\ol{u}=\ol{v}$ if and only if there are words $w_1,\dots,w_\ell$ such that $w_1=u$,
$w_i\equiv w_{i+1}$ for all $1\leq i<\ell$, and $w_\ell=v$. By repeated applications of Lemma \ref{lem1}, we
obtain minimal r-factorisations
\begin{equation}
w_j = p_1^{(j)}\dots p_m^{(j)} \label{wj-fact}
\end{equation}
for $1\leq j\leq\ell$ such that $\ol{p_i^{(j)}}\;\D\;\ol{p_i^{(j')}}$ for all $1\leq i\leq m$
and $1\leq j,j'\leq\ell$ (we assume that $p_i^{(1)}=p_i$). In particular, for all such $i$ we have $\ol{p_i}\;\D\;
\ol{p_i^{(\ell)}}$ and
$$
v = p_1^{(\ell)}\dots p_m^{(\ell)} = q_1\dots q_s.
$$
Now Lemma \ref{lem2} implies that $s=m$ and $\ol{p_i^{(\ell)}}\;\D\;\ol{q_i}$ for all $1\leq i\leq m$. Furthermore,
by the same lemma we have $\ol{p_1}\;\R\;\ol{p_i^{(j)}}\;\R\;\ol{q_1}$ and $\ol{p_m}\;\L\;\ol{p_m^{(j)}}\;\L\;\ol{q_m}$
for all $1\leq j\leq\ell$, and the theorem is proved.
\end{proof}

Another way of expressing the previous theorem is that for a fixed element of $\ig{\cE}$, for any word $w\in E^+$ such that
$\ol{w}$ equals the element in question and any minimal r-factorisation $w=p_1\dots p_m$ the sequence
$(D_{\ol{p_i}},\dots,D_{\ol{p_m}})$ of regular $\D$-classes of $\ig{\cE}$ is uniquely determined by the considered element.
We call this sequence the \emph{$\D$-fingerprint} of the element. So, for any two words $u,v\in E^+$ to stand any chance
of representing the same element of $\ig{\cE}$ we must first have that $\ol{u}$ and $\ol{v}$ have the same $\D$-fingerprint.
In particular, Lemma \ref{lem0} implies that if $u=p_1\dots p_m$, $v=q_1\dots q_m$ are minimal r-factorisations and we have
$(p_1,\dots,p_m)\sim (q_1,\dots,q_m)$ then $\ol{p_i}\;\D\;\ol{q_i}$ for all $1\leq i\leq m$. In fact, we have the following
description of the word problem of $\ig{\cE}$ via the $\sim$ relation and minimal r-factorisations.

\begin{thm}\label{wp-sim}
Let $u,v\in E^+$. The following conditions are equivalent:
\begin{itemize}
\item[(1)] $\ol{u}=\ol{v}$;
\item[(2)] there exist minimal r-factorisations of $u$ and $v$, respectively, having the same number of factors, say
$u=p_1\dots p_m$ and $v=q_1\dots q_m$ for some $m\geq 1$, such that 
$$
(p_1,\dots,p_m) \sim (q_1,\dots,q_m);
$$
\item[(3)] there exists an integer $m\geq 1$ such that all minimal r-factorisations of $u$ and $v$, respectively, have
precisely $m$ factors, and whenever $u=p_1\dots p_m$ and $v=q_1\dots q_m$ are such factorisations we have
$$
(p_1,\dots,p_m) \sim (q_1,\dots,q_m).
$$
\end{itemize}
\end{thm}

\begin{proof}
(1)$\Ra$(2) Let $u=p_1\dots p_m$ and $v=q_1\dots q_s$ be arbitrary (but fixed) minimal r-facto\-ri\-sa\-tions of $u$ and $v$, 
respectively. By Theorem \ref{D-fing}, we have $m=s$. Furthermore, there are words $w_1,\dots,w_\ell$ such that $w_1=u$, 
$w_j\equiv w_{j+1}$ for all $1\leq j<\ell$, and $w_\ell=v$. Just as in the proof of the previous theorem, repeated applications 
of Lemma \ref{lem1} yield minimal r-factorisations \eqref{wj-fact} of the words $w_j$ such that 
$$
(p_1,\dots,p_m) \approx (p_1^{(2)},\dots,p_m^{(2)}) \approx \dots \approx (p_1^{(\ell)},\dots,p_m^{(\ell)}).
$$
Now $p_1^{(\ell)}\dots p_m^{(\ell)}$ and $q_1\dots q_m$ are minimal r-factorisations of the same word, $v$, from which
it is immediate that
$$
(p_1^{(\ell)},\dots,p_m^{(\ell)}) \sim (q_1,\dots,q_m),
$$
namely, if $(1,\alpha_2,\dots,\alpha_m)$ and $(1,\beta_2,\dots,\beta_m)$ coordinatise the latter two factorisations, then
Lemma \ref{lem2} implies that $1<\alpha_2,\beta_2<\dots<\alpha_m,\beta_m$ and the relocation of the letters from the subword
between positions $\min(\alpha_i,\beta_i)$ and $\max(\alpha_i,\beta_i)-1$ from one factor to its neighbour yields the above
relation. Therefore, we have $(p_1,\dots,p_m) \sim (q_1,\dots,q_m)$.

(2)$\Ra$(3) Let $u=p_1\dots p_m$ and $v=q_1\dots q_m$ be the minimal r-factorisations provided by the condition (2), and let
$u=p_1'\dots p_r'$ and $v=q_1'\dots q_s'$ be arbitrary minimal r-factorisations of $u$ and $v$, respectively. Then Theorem
\ref{D-fing} implies that $r=m=s$. Just as in the previous paragraph we conclude that $(p_1',\dots,p_m')\sim(p_1,\dots,p_m)$
and $(q_1,\dots,q_m)\sim(q_1',\dots,q_m')$ (as these pairs of sequences represent minimal r-factorisations of the same word),
whence the assumption $(p_1,\dots,p_m) \sim (q_1,\dots,q_m)$ from (2) yields the required conclusion $(p_1',\dots,p_m') \sim 
(q_1',\dots,q_m')$.

(3)$\Ra$(1) This implication follows immediately from Lemma \ref{lem0}.
\end{proof}

\begin{rmk}\label{wp-eff}
Let us briefly explain why the previous theorem effectively reduces the word problem of $\ig{\cE}$ to the
question of computability of the relation $\sim$ on finite sequences of words over $E$ comprising minimal r-factorisations
of words with respect to $\cE$. The reason for this lies in the certainty of finding at least one minimal r-factorisation
of any word $w$ algorithmically, and, in addition, locating a seed for each of the factors. Indeed, since $|w|$ is finite,
there are only finitely many factorisations of the word $w$. For any such factorisation $w=w_1\dots w_m$ we can determine
whether it is an r-factorisation by testing (using Theorem \ref{thm:DGR}(iii)) whether each $\ol{w_i}$ is regular. In this way
we can generate the list of all r-factorisations of $w$, and since the list is finite we can pick a minimal one.
Hence, given two words $u,v$, there is an algorithm that returns a minimal r-factorisation for each of these
two words and now, by Theorem \ref{wp-sim}, $\ol{u}=\ol{v}$ is equivalent to the $\sim$-relatedness of the
sequences of factors of the computed minimal r-factorisations.

Furthermore,
if $w=w_1\dots w_m$ is a minimal r-factorisation, we can run through all factorisations $w_i=u_ie_iv_i$ using
Theorem \ref{thm:DGR}(i) to determine if they satisfy the conditions of Theorem \ref{thm:DGR}(ii) and thus find a seed $e_i$
for each $w_i$. This shows that given a word $w$, the $\D$-fingerprint of $\ol{w}$ is algorithmically computable,
as this is now simply the sequence $(D_{\ol{e_1}},\dots,D_{\ol{e_m}})$.
\end{rmk}

\section{Rees matrix coordinatisation of regular factors}
\label{sec:Rees}

Paralleling the notion of a minimal r-factorisation of a word in $E^+$ with respect to $\cE$, we introduce the notion
of an \emph{irreducible r-sequence} $(\mathbf{r}_1,\dots,\mathbf{r}_m)$ of regular elements of $\ig{\cE}$: these
are characterised by the property that no nontrivial product of consecutive elements is regular. It is
easy to see that $(\mathbf{r}_1,\dots,\mathbf{r}_m)$ is an irreducible r-sequence if and only if for some
(equivalently, any) $p_1,\dots,p_m\in E^+$ such that $\mathbf{r}_i=\ol{p_i}$ for all $1\leq i\leq m$ we have
that $(p_1,\dots,p_m)$ is a minimal r-factorisation of some word $w$ such that
$\ol{w}=\mathbf{r}_1\dots\mathbf{r}_m$.

Now, we can pass from the `syntactic' relation $\sim$ defined on (certain) sequences of words to the relation
$\simeq$ defined on the set of irreducible r-sequences, defined by
$$
(\mathbf{r}_1,\dots,\mathbf{r}_m) \simeq (\mathbf{s}_1,\dots,\mathbf{s}_m)
$$
if and only if $(p_1,\dots,p_m) \sim (q_1,\dots,q_m)$ for some words $p_i,q_i\in E^+$ such that
$\mathbf{r}_i=\ol{p_i}$ and $\mathbf{s}_i=\ol{q_i}$ for all $1\leq i\leq m$. This is well-defined since 
it does not depend on the choice of the words $p_i,q_i$, because if $p_i',q_i'\in E^+$ are
words such that $\ol{p_i'}=\ol{p_i}$ and $\ol{q_j'}=\ol{q_j}$ for all $1\leq i\leq m$ then clearly
$(p_1,\dots,p_m)\sim (p_1',\dots,p_m')$ and $(q_1,\dots,q_m)\sim (q_1',\dots,q_m')$, so we have that
$(p_1,\dots,p_m) \sim (q_1,\dots,q_m)$ if and only if $(p_1',\dots,p_m') \sim (q_1',\dots,q_m')$.
Therefore, we can conclude that the word problem of $\ig{\cE}$ relies on the computability of the relation
$\simeq$ on the set of irreducible r-sequences, since we have (as a direct corollary of Theorem \ref{wp-sim})
$\mathbf{r}_1\dots\mathbf{r}_m=\mathbf{s}_1\dots\mathbf{s}_m$ if and only if $(\mathbf{r}_1,\dots,\mathbf{r}_m)
\simeq (\mathbf{s}_1,\dots,\mathbf{s}_m)$. Also recall that, by our earlier remarks,
$(\mathbf{r}_1,\dots,\mathbf{r}_m) \simeq (\mathbf{s}_1,\dots,\mathbf{s}_m)$ implies
$\mathbf{r}_i\;\D\;\mathbf{s}_i$ for all $1\leq i\leq m$ (and so the $\D$-fingerprints of both products
$\mathbf{r}_1\dots\mathbf{r}_m$ and $\mathbf{s}_1\dots\mathbf{s}_m$ are the same:
$(D_{\mathbf{r}_1},\dots,D_{\mathbf{r}_m})$).

To connect the relation $\simeq$ (and thus the word problem of $\ig{\cE}$) to the maximal subgroups of $\ig{\cE}$,
we must find a different representation of a regular element of $\ig{\cE}$ (other than by a word over $E$).
This is the core of the idea contained in \cite[Section 4]{DGR}, and we are going to review here the corresponding
transformation of the problem, taking special care that each step along the way is effectively algorithmic.
Namely, we already know from Lemma \ref{lem:DJ} that in our setting of a finite $E$, for any regular $\D$-class $D$ of $\ig{\cE}$
the $\J$-class containing $D$ coincides with $D$, and, furthermore, the corresponding principal factor $D^0$
is completely 0-simple and thus isomorphic to a suitable Rees matrix semigroup $\mathcal{M}^0[G;I,\Lambda;P]$.
By Remark \ref{rem:RLD}, the sets $I,\Lambda$ can be taken to be the sets enumerating the $\R$-classes (resp.\
$\L$-classes) of the corresponding $\D$-class of $\cE$, and thus a byproduct of the input data.

Furthermore, as noted in Theorems 3.10 and 4.2 of \cite{DGR}, the group $G$ and the sandwich matrix $P$ are also
known, and a finite presentation for $G$ (as well as $P$) can be algorithmically computed from $\cE$.
We will not repeat here the presentation (referring instead either to Theorem 5 of the seminal paper \cite{GR1}
or \cite[Theorem 4.2]{DGR}); we will remain content with saying that this presentation is given in terms of
generators $f_{i\lambda}$ such that $R_i\cap L_\lambda$ contains an idempotent ($e_{i\lambda}$ in $\cE$ or,
equivalently, $\ol{e_{i\lambda}}$ in $\ig{\cE}$) -- in which case $p_{\lambda i}= f_{i\lambda}^{-1}$ and
$p_{\lambda i}=0$ otherwise -- and the relators depend on three configurations in $\cE$ called
the \emph{anchors} (prompting some of the generators to be $=1$), the \emph{Schreier system} (yielding certain
equalities between the generators), and \emph{singular squares}. Since we are going to need it later, we describe
this latter, third group of relations: these are of the form
\begin{equation}
f_{i\lambda}^{-1}f_{i\mu} = f_{j\lambda}^{-1}f_{j\mu}, \label{sing-sq}
\end{equation}
where the idempotents $e_{i\lambda},e_{i\mu},e_{j\lambda},e_{j\mu}$ form a \emph{singular square} in $\cE$.
This means that there exists an idempotent $f\in E$ such that one of the following sets of conditions hold:
\begin{itemize}
\item[(a)] $f\ast e_{i\lambda}=e_{i\lambda}$, $f\ast e_{j\lambda}=e_{j\lambda}$, $e_{i\lambda}\ast f=e_{i\mu}$,
$e_{j\lambda}\ast f=e_{j\mu}$;
\item[(b)] $e_{i\lambda}\ast f=e_{i\lambda}$, $e_{j\lambda}\ast f=e_{j\lambda}$, $f\ast e_{i\lambda}=e_{j\lambda}$,
$f\ast e_{i\mu}=e_{j\mu}$.
\end{itemize}
An important observation to be utilised later is that the maximal subgroup of $D$ is free whenever there are
no singular squares in $D$. In particular, this will hold whenever $D$ is a maximal $\D$-class.

Let $\varphi:D^0 \to \mathcal{M}^0[G;I,\Lambda;P]$ be an isomorphism. If $L_D\subseteq E^+$ denotes the language
of all words over $E$ that represent an element of $D$, then $\varphi$ naturally induces a mapping $\ol\varphi:
L_D \to \mathcal{M}^0[G;I,\Lambda;P]$ defined by $\ol\varphi(w)=\varphi(\ol{w})$. In order to effectively
`Rees matrix coordinatise' $D$, it is our main goal here to show that there is an algorithm that computes one
such mapping $\ol\varphi$. As a starting point, note that for any idempotent $e_{i\lambda}\in D$ we necessarily
have
$$
\ol\varphi(e_{i\lambda}) = \varphi(\ol{e_{i\lambda}}) = (i,p_{\lambda i}^{-1},\lambda) = (i,f_{i\lambda},\lambda).
$$
The general idea is that if $w\in L_D$ then by Theorem \ref{thm:DGR}(ii) $w$ must have a seed, that is, a
factorisation $w=uev$ such that $\ol{e}\in D$ and $\ol{ev}\;\R\;\ol{e}\;\L\;\ol{ue}$; furthermore, as noted in
Remark \ref{wp-eff}, there is an algorithm which identifies at least one such factorisation. Since $\ol{e}\in D$,
the letter $e$ must, in fact, be $e_{i\lambda}$ for some $i\in I$, $\lambda\in\Lambda$, and so we already
know the `coordinates' $\varphi(\ol{e_{i\lambda}})$ of $e_{i\lambda}$ as above. Then, we will study the way
in which idempotents $\ol{e}$ act (from the right and from the left) on elements of $D$ corresponding to certain
Rees matrix triples, and obtain explicit formulae for these actions. Finally, we shall successively apply
these formulae for idempotents represented by letters of $u,v$ to compute $\ol\varphi(w)=\ol\varphi(uev)$.

In the sense just described, the following result (representing a significant generalisation of \cite[Lemma 8.9]{DGR})
is of key importance. Below, for a partial function $\alpha$, let $\fix(\alpha)$ denotes the set of fixed points of $\alpha$
(which in the case when the map $\alpha$ is idempotent coincides with the image of $\alpha$).

\begin{pro}\label{action}
Let $D$ be a regular $\D$-class of $\ig{\cE}$ and let $$\varphi:D^0 \to \mathcal{M}^0[G;I,\Lambda;P]$$ be an isomorphism. 
Let $\mathbf{r}=\varphi^{-1}(i,g,\lambda)\in D$ and $e\in E$.
\begin{itemize}
	
	\item[(1)] If $[\varphi^{-1}(i,g,\lambda)]\ol{e}\in D$ then $[\varphi^{-1}(i,g,\lambda)]\ol{e}\;\R\;
	\varphi^{-1}(i,g,\lambda)$, and if
	$\omega_{\lambda,e}\in G$ is such that $[\varphi^{-1}(i,g,\lambda)]\ol{e}=
	\varphi^{-1}(i,g\omega_{\lambda,e},\lambda')$ then for any $j\in I$ and $g'\in G$ we have
	$$
	[\varphi^{-1}(j,g',\lambda)]\ol{e} = \varphi^{-1}(j,g'\omega_{\lambda,e},\lambda').
	$$
	In this sense, $e$ induces an idempotent partial transformation $\tau_e$ (acting on the right) on $\Lambda$ given by
	$\lambda\mapsto\lambda'$ whenever $[\varphi^{-1}(i,g,\lambda)]\ol{e}=\varphi^{-1}(i,h,\lambda')$ for some
	$i\in I$ and $g,h\in G$. The partial transformation $\tau_e$ is effectively computable from $\cE$.
	
	\item[(2)] If $\ol{e}[\varphi^{-1}(i,g,\lambda)]\in D$ then $\ol{e}[\varphi^{-1}(i,g,\lambda)]\;\L\;
	\varphi^{-1}(i,g,\lambda)$, and if
	$\omega_{e,i}\in G$ is such that $\ol{e}[\varphi^{-1}(i,g,\lambda)]=
	\varphi^{-1}(i',\omega_{e,i}g,\lambda)$ then for any $\mu\in\Lambda$ and $g'\in G$ we have
	$$
	\ol{e}[\varphi^{-1}(i,g',\mu)] = \varphi^{-1}(i',\omega_{e,i}g',\mu).
	$$
	In this sense, $e$ induces an idempotent partial transformation $\sigma_e$ (acting on the left) on $I$ given by
	$i\mapsto i'$ whenever $\ol{e}[\varphi^{-1}(i,g,\lambda)]=\varphi^{-1}(i',h,\lambda)$ for some
	$\lambda\in\Lambda$ and $g,h\in G$. The partial transformation $\sigma_e$ is effectively computable from $\cE$.
	
	\item[(3)] The partial map $\tau_e$ is non-empty if and only if $\sigma_e$ is non-empty. If
	$\mathbf{r}\ol{e}\in D$ then $\fix(\sigma_e)\neq\es$, and for arbitrary $i_0\in\fix(\sigma_e)$ we have
	$\omega_{\lambda,e}=f_{i_0,\lambda}^{-1}f_{i_0,\lambda\tau_e}$ and thus
	$$
	[\varphi^{-1}(i,g,\lambda)]\ol{e}=\varphi^{-1}(i,gf_{i_0,\lambda}^{-1}f_{i_0,\lambda\tau_e},\lambda\tau_e).
	$$
	Dually, if $\ol{e}\mathbf{r}\in D$ then $\fix(\tau_e)\neq\es$, and for arbitrary
	$\lambda_0\in\fix(\tau_e)$ we have $\omega_{e,i}=f_{\sigma_e i,\lambda_0}f_{i,\lambda_0}^{-1}$ and so
	$$
	\ol{e}[\varphi^{-1}(i,g,\lambda)]=\varphi^{-1}(\sigma_e i,f_{\sigma_e i,\lambda_0}f_{i,\lambda_0}^{-1}g,\lambda).
	$$
	
\end{itemize}
\end{pro}

\begin{proof}
Let us start by proving (1). We are given that $[\varphi^{-1}(i,g,\lambda)]\ol{e}$ is $\D$-related
to $\varphi^{-1}(i,g,\lambda)$; in particular, $[\varphi^{-1}(i,g,\lambda)]\ol{e}$ is regular. So, by Lemma \ref{lem:FG} we can write
$\varphi^{-1}(i,g,\lambda)$ as a product of idempotents from its $\D$-class $D$:
$$
\mathbf{r}=\varphi^{-1}(i,g,\lambda) = \ol{e_{i\lambda_1}}\,\ol{e_{i_2\lambda_2}}\dots\ol{e_{i_n\lambda}},
$$
and, by Theorem \ref{thm:DGR}(ii), the product on the right-hand side has at least one seed, so that for some
$1\leq k\leq n$ we have
$$
\ol{e_{i_k\lambda_k}}\dots\ol{e_{i_n\lambda}}\,\ol{e}\;\R\;\ol{e_{i_k\lambda_k}}\;\R\;
\ol{e_{i_k\lambda_k}}\dots\ol{e_{i_n\lambda}}.
$$
Since $\R$ is a left congruence, by multiplying by $\ol{e_{i\lambda_1}}\dots\ol{e_{i_{k-1}\lambda_{k-1}}}$ from the left
we get $[\varphi^{-1}(i,g,\lambda)]\ol{e}\;\R\;\varphi^{-1}(i,g,\lambda)$, as required. Therefore, $[\varphi^{-1}(i,g,\lambda)]\ol{e}=
\varphi^{-1}(i,h,\lambda')$ for some $h\in G$ and $\lambda'\in\Lambda$. Define $\omega_{\lambda,e}=g^{-1}h$, and let
$j\in I$ and $g'\in G$ be arbitrary. Choose $k\in I$ and $\mu\in \Lambda$ such that $p_{\lambda k}\neq 0$ and $p_{\mu i}\neq 0$.
Since
$$
\varphi^{-1}(j,g',\lambda) = \varphi^{-1}(j,g',\lambda)\varphi^{-1}(k,f_{k\lambda}g^{-1}f_{i\mu},\mu)
\varphi^{-1}(i,g,\lambda)
$$
it follows that
$$
[\varphi^{-1}(j,g',\lambda)]\ol{e} = \varphi^{-1}(j,g',\lambda)\varphi^{-1}(k,f_{k\lambda}g^{-1}f_{i\mu},\mu)\varphi^{-1}
(i,h,\lambda')=\varphi^{-1}(j,g'\omega_{\lambda,e},\lambda'),
$$
just as claimed.

Furthermore, the partial transformation $\tau_e$ is computable from $\cE$ because by results already proved,
to compute $\lambda'=\lambda\tau_e$ we can start from $\ol{e_{i\lambda}}=\varphi^{-1}(i,f_{i\lambda},\lambda)$ whence
$\ol{e_{i\lambda}e}\in D$ and so $\ol{e_{i\lambda}e}\;\R\;\ol{e_{i\lambda}}$. Now Theorem \ref{thm:DGR}(i) guarantees that
there is an algorithm computing an idempotent $\ol{f}\in D$ such that $\ol{e_{i\lambda}e}\;\L\;\ol{f}$, and thus we can
identify the $\L$-class of $\ol{e_{i\lambda}e}$, which is exactly $L_{\ol{f}}=\varphi^{-1}(L_{\lambda'})$.

Since (2) is completely left-right dual to (1), we immediately move on to proving (3). To see this, assume that $\tau_e$ is
non-empty, so that $$[\varphi^{-1}(i,g,\lambda)]\ol{e}=\varphi^{-1}(i,g\omega_{\lambda,e},\lambda')$$ is an element of $D$
for some $i\in I$, $g\in G$ and $\lambda \in \Lambda$ (with $\lambda'=\lambda\tau_e$). We recall Fitz-Gerald's procedure
\cite[Lemma 2]{FG} (see also \cite[page 236, Exercise 12]{HoBook}) of turning
the product $\mathbf{r}\ol{e}=[\varphi^{-1}(i,g,\lambda)]\ol{e} = \ol{e_{i\lambda_1}}\dots \ol{e_{i_n\lambda}}\,\ol{e}$
into a product of idempotents from $D$ (note that $\ol{e}$ might be not in $D$, thus `generating' the necessity for rewriting
this product). Upon choosing an inverse $\mathbf{s}'$ of $\mathbf{s}=\mathbf{r}\ol{e}$, this consists in forming the products
$$
\mathbf{f}_k = \ol{e_{i_k\lambda_k}}\dots\ol{e_{i_n\lambda}}\,\ol{e}\mathbf{s}'\ol{e_{i\lambda_1}}\dots\ol{e_{i_k\lambda_k}}
$$
for all $1\leq k\leq n$, and
$$
\mathbf{f}_{n+1} = \ol{e}\mathbf{s}'\ol{e_{i\lambda_1}}\dots\ol{e_{i_n\lambda}}\,\ol{e},
$$
whence $\mathbf{s}=\mathbf{f}_1\dots\mathbf{f}_n\mathbf{f}_{n+1}$. We also know that each of these products is an idempotent,
and since for $1\leq k\leq n$, $\mathbf{f}_k$ is clearly both $\R$- and $\L$-related to $\ol{e_{i_k\lambda_k}}$, we must
have $\mathbf{f}_k=\ol{e_{i_k\lambda_k}}$; so, $\mathbf{r}=\mathbf{f}_1\dots\mathbf{f}_n$ and $\mathbf{s}=\mathbf{r}
\mathbf{f}_{n+1}$. Furthermore, as $\mathbf{s}\ol{e}=\mathbf{s}$, $\mathbf{f}_{n+1}=\ol{e}(\mathbf{s}'\mathbf{s})$ is 
an idempotent from $D$, just as $\mathbf{s}'\mathbf{s}$ is an idempotent, having the form 
$\varphi^{-1}(j,f_{j\lambda'},\lambda')$ for some $j\in I$ such that $R_j\cap L_{\lambda'}$
contains an idempotent. (Indeed, $\mathbf{s}\ol{e}=\mathbf{r}\ol{e}\,\ol{e}=\mathbf{r}\ol{e}=\mathbf{s}$ implying
$(\mathbf{s}'\mathbf{s})\ol{e}=\mathbf{s}'\mathbf{s}$, so that $\{e_{j\lambda'},e\}$ is a basic pair.) This shows (by part (2))
that $\sigma_e$ cannot be empty, as it is defined on $j$. The converse implication is dual.

Finally, let us write $i_0=\sigma_e j$. We now have $\mathbf{f}_{n+1}=\varphi^{-1}(i_0,f_{i_0\lambda'},\lambda')$, and hence
$$
\varphi^{-1}(i,g\omega_{\lambda,e},\lambda') = [\varphi^{-1}(i,g,\lambda)]\ol{e} = [\varphi^{-1}(i,g,\lambda)]\mathbf{f}_{n+1} =
\varphi^{-1}(i,gf_{i_0,\lambda}^{-1}f_{i_0,\lambda'},\lambda'),
$$
which immediately implies the expression for $\omega_{\lambda,e}$. Notice that by the defining relations \eqref{sing-sq} of $G$,
the actual choice of $i_0$ (and thus of $j$) is immaterial: if $i'_0$ is another image (fixed) point of $\sigma_e$ then $e$ takes
the role of $f$ in case (a) of the definition of singular squares above, singularising the square of idempotents
$e_{i_0\lambda},e_{i_0\lambda'},e_{i'_0\lambda},e_{i'_0\lambda',}$, so we have that $f_{i_0,\lambda}^{-1}f_{i_0,\lambda'}=
f_{i'_0,\lambda}^{-1}f_{i'_0,\lambda'}$ holds in $G$. The formula for $\omega_{e,i}$ follows dually.
\end{proof}

\begin{rmk}
The idempotent $\ol{e}$ from the previous proposition must be from a $\D$-class that is (non-strictly) $\J$-above $D$, the
$\D$-class of $\varphi^{-1}(i,g,\lambda)$.
\end{rmk}

We can now present the result that describes the effective transformation of a word representing an element of a regular $\D$-class
$D$ of $\ig{\cE}$ into a Rees matrix triple.

\begin{thm}\label{thm:Rees}
Let $w\in E^+$ be such that $\ol{w}\in D$ is a regular element of $\ig{\cE}$, and let $w=ue_{i\lambda}v$ be a factorisation of
$w$ determining the position of a seed of $w$, so that $\ol{e_{i\lambda}}\in D$ and $\ol{ue_{i\lambda}}\;\L\;\ol{e_{i\lambda}}
\;\R\;\ol{e_{i\lambda}v}$. Furthermore, let $u=g_1\dots g_k$ and $v=h_1\dots h_l$ for $g_p,h_q\in E$, $1\leq p\leq k$,
$1\leq q\leq l$. Then
$$
\ol\varphi(w) = \left(i_1,f_{i_1\mu_1}^{}f_{i_2\mu_1}^{-1}\dots f_{i_k\mu_k}^{}f_{i\mu_k}^{-1}f_{i\lambda}^{}
f_{j_1\lambda}^{-1}f_{j_1\lambda_1}^{}\dots f_{j_l\lambda_{l-1}}^{-1}f_{j_l\lambda_l}^{},\lambda_l\right),
$$
where $\lambda_1=\lambda\tau_{h_1}$, $\lambda_{q+1}=\lambda_q\tau_{h_{q+1}}$ for $1\leq q<l$, $j_q\in\fix(\sigma_{h_q})$ for
$1\leq q\leq l$, $i_k=\sigma_{g_k}i$, $i_p=\sigma_{g_p}i_{p+1}$ for $1\leq p<k$, and $\mu_p\in\fix(\tau_{g_p})$ for
$1\leq p\leq k$.
\end{thm}

\begin{proof}
We have $\ol{e_{i\lambda}v}=[\varphi^{-1}(i,f_{i\lambda},\lambda)]\ol{h_1}\dots\ol{h_l}$, and the right-hand side is now seen to
be equal to
$$
\varphi^{-1}\left(i,f_{i\lambda}^{}f_{j_1\lambda}^{-1}f_{j_1\lambda_1}^{}f_{j_2\lambda_1}^{-1}\dots f_{j_l\lambda_{l-1}}^{-1}
f_{j_l\lambda_l}^{},\lambda_l\right)
$$
by repeated applications of Proposition \ref{action}(1) and the first formula in (3). In each step, $\lambda_q=\lambda_{q-1}\tau_{h_q}$
(here we set $\lambda_0=\lambda$) is defined because
$$\ol{e_{i\lambda}v}\;\R\;\ol{e_{i\lambda}h_1\dots h_{q-1}h_q}\;\R\;\ol{e_{i\lambda}h_1\dots h_{q-1}}\;\R\;\ol{e_{i\lambda}}$$
and, consequently, the idempotent partial map $\sigma_{h_q}$ is not empty, thus $\fix(\sigma_{h_q})\neq\es$. By left-right duality,
Proposition \ref{action}(2) and the second formula in (3) imply
$$
\ol{ue_{i\lambda}} = \varphi^{-1}\left(i_1,f_{i_1\mu_1}^{}f_{i_2\mu_1}^{-1}\dots f_{i_k\mu_{k-1}}^{-1}f_{i_k\mu_k}^{}f_{i\mu_k}^{-1}
f_{i\lambda}^{},\lambda\right),
$$
and analogous statements justifying the existence of $i_p$ and $\mu_p$ hold. The theorem now follows as a consequence of
$\ol\varphi(w)=\varphi(\ol{w})=\varphi(\ol{ue_{i\lambda}v})=\varphi(\ol{ue_{i\lambda}})\varphi(\ol{e_{i\lambda}v})$.
\end{proof}

\begin{rmk}
\begin{itemize}
\item[(i)] The previous theorem shows that there is an algorithm which, given a word $w\in E^+$ such that $\ol{w}\in D$ is regular,
computes $\ol\varphi(w)=\varphi(\ol{w})$. This follows because all the parameters involved in the formula depend solely on the
knowledge of partial mappings $\sigma_e,\tau_e$, $e\in E$, acting on the index sets of $\R$-classes (resp.\ $\L$-classes) of $D$,
and we have shown in Proposition \ref{action} that there exist algorithms computing these mapping from the given biordered set $\cE$.
\item[(ii)] The previous theorem also yields an effective version of Lemma \ref{lem:FG} in the considered context (that is, $\ol{w}
\in D$), turning a word $w$ whose value is in $D$ into a word over $E_D$, the set of all $e\in E$ such that $\ol{e}\in D$, because
it is straightforward to verify that
$$
\ol\varphi(w) = \ol\varphi(e_{i_1\mu_1}\dots e_{i_k\mu_k}e_{i\lambda}e_{j_1\lambda_1}\dots e_{j_l\lambda_l}),
$$
with all the parameters as in the previous theorem. Therefore,
$$\ol{w} = \ol{e_{i_1\mu_1}\dots e_{i_k\mu_k}e_{i\lambda}e_{j_1\lambda_1}\dots e_{j_l\lambda_l}}$$
holds in $\ig{\cE}$.
\end{itemize}
\end{rmk}

To simplify the notation in the remainder of the paper, we are going to abuse it slightly and simply use the triple $(i,g,\lambda)$
as short-hand for the regular element $\varphi^{-1}(i,g,\lambda)\in D$. In dropping $\varphi$ from the notation we assume from
now on that for each regular $\D$-class of $\ig{\cE}$ we have fixed an isomorphism from its principal factor to the corresponding
Rees matrix semigroup, namely the one given by Theorem \ref{thm:Rees}.

We now return to considering the relation $\simeq$ on the set of irreducible r-sequ\-en\-ces $(\mathbf{r}_1,\dots,
\mathbf{r}_m)$ (of regular elements of $\ig{\cE}$). Assume that the $\D$-fingerprint of $\mathbf{r}_1\dots
\mathbf{r}_m$ is $(D_1,\dots,D_m)$, so that $\mathbf{r}_k\in D_k$ for all $1\leq k\leq m$. (Note that these
$\D$-classes are not all necessarily different.) Each of these $\D$-classes is `Rees matrix coordinatised'
as described in the present section, so that elements of $D_k$, $1\leq k\leq m$, are identified with triples of the form
$(i_k,g_k,\lambda_k)\in I^{(k)}\times G^{(k)}\times \Lambda^{(k)}$. Our aim is to describe the $\simeq$ relation between
tuples of these Rees matrix triples corresponding to irreducible r-sequences.

Bearing in mind the definitions of $\sim$ and $\simeq$, we have that the latter relation is the reflexive-transitive
closure of the relation $\lra$, where $(\mathbf{r}_1,\dots,\mathbf{r}_m) \lra (\mathbf{s}_1,\dots,\mathbf{s}_m)$
if and only if either
\begin{enumerate}
\item[(i)] $\mathbf{r}_k=\mathbf{s}_k\ol{e}$ and $\mathbf{s}_{k+1}=\ol{e}\mathbf{r}_{k+1}$ for some $e\in E$ and
$k<m$, and $\mathbf{r}_l=\mathbf{s}_l$ for all $l\not\in\{k,k+1\}$, or
\item[(ii)] $\mathbf{s}_k=\mathbf{r}_k\ol{e}$ and $\mathbf{r}_{k+1}=\ol{e}\mathbf{s}_{k+1}$ for some $e\in E$ and
$k<m$, and $\mathbf{r}_l=\mathbf{s}_l$ for all $l\not\in\{k,k+1\}$.
\end{enumerate}
This immediately translates to $$((i_1,g_1,\lambda_1),\dots,(i_m,g_m,\lambda_m)) \lra ((j_1,h_1,\mu_1),\dots,
(j_m,h_m,\mu_m))$$ if and only if either
\begin{enumerate}
\item[(i)] $(i_k,g_k,\lambda_k)=(j_k,h_k,\mu_k)\ol{e}$ and $(j_{k+1},h_{k+1},\mu_{k+1})=\ol{e}(i_{k+1},g_{k+1},
\lambda_{k+1})$ for some $e\in E$ and $k<m$, and $(i_l,g_l,\lambda_l)=(j_l,h_l,\mu_l)$ for all
$l\not\in\{k,k+1\}$, or
\item[(ii)] $(j_k,h_k,\mu_k)=(i_k,g_k,\lambda_k)\ol{e}$ and $(i_{k+1},g_{k+1},\lambda_{k+1})=\ol{e}(j_{k+1},h_{k+1},
\mu_{k+1})$ for some $e\in E$ and $k<m$, and $(i_l,g_l,\lambda_l)=(j_l,h_l,\mu_l)$ for all $l\not\in\{k,k+1\}$.
\end{enumerate}
Note that the form of the action of the idempotent $\ol{e}$ on a triple is exactly the one described in Proposition \ref{action}.

Therefore, we can conclude this section by saying that in the course of establishing $((i_1,g_1,\lambda_1),\dots,(i_m,g_m,\lambda_m))
\simeq ((j_1,h_1,\mu_1),\dots,(j_m,h_m,\mu_m))$ for two irreducible r-sequences of (elements represented by) triples, a single step
involves picking two adjacent triples in the sequence, say $(i,g,\lambda)\in D_k$ and $(j,h,\mu)\in D_{k+1}$, and then either
performing an inverse action of an idempotent $\ol{e}$ (that is $\J$-above both $D_k$ and $D_{k+1}$) on the first factor from the right
and then applying $\ol{e}$ on the second factor from the left, or the other way round: performing an inverse action of $\ol{e}$ on the
second factor from the left and applying $\ol{e}$ on the first factor from the right. Either way, the described transformation prompts
the pair $(\lambda,j)$ of `inner indices' to change (depending on the maps $\sigma_e,\tau_e$), and also the group components $g,h$
get multiplied from the right and left, respectively, by certain factors belonging to their corresponding groups $G^{(k)},G^{(k+1)}$.
This observation suggests that `in between' any two regular $\D$-classes $D_k,D_{k+1}$ that are adjacent in a fixed $\D$-fingerprint,
there exists an associated transition system with the set of states $\Lambda^{(k)}\times I^{(k+1)}$ whose transitions are labelled
both by elements of $E$ (the intervening idempotents) and pairs from $G^{(k)}\times G^{(k+1)}$ (the right and left multipliers of
the group parts). These systems will be called \emph{contact automata} (of two regular $\D$-classes), and they will be introduced and
studied in the next section; eventually, they will bring about the required group-theoretical interpretation of the word problem
of $\ig{\cE}$.

\section{Contact automata}
\label{sec:contact}

Let $D_1,D_2$ be regular $\D$-classes of $\ig{\cE}$. Just as argued in the previous section, we assume that $D_k$, $k=1,2$, are
coordinatised by sets $I_k$ and $\Lambda_k$, indexing their $\R$- and $\L$-classes, respectively. Furthermore, let $G_k$ be
the maximal subgroup of $D_k$ defined by the presentation from \cite[Theorem 5]{GR1} (or \cite[Theorem 4.2]{DGR}) in terms
of generators $f_{i\lambda}^{(k)}$, where $i\in I_k$, $\lambda\in\Lambda_k$, are such that in $D_k$, $R_i^{(k)}\cap L_\lambda^{(k)}$
contains an idempotent (and thus is a group itself, isomorphic to $G_k$). Of course, for our purpose we are going to
be interested only in pairs of $\D$-classes $D_1,D_2$ such that $\mathbf{r}_1\mathbf{r}_2$ is not regular for some
$\mathbf{r}_k\in D_k$, $k=1,2$, but in principle the following definition works for an arbitrary (ordered) pair of $\D$-classes.
So, we define the \emph{contact automaton of $\D$-classes $D_1,D_2$} to be the ``transition system''
$$
\mathcal{A}(D_1,D_2) = (\Lambda_1\times I_2, E, \Delta, \ell),
$$
where $\Lambda_1\times I_2$ is the set of states, $E$ is the alphabet, $\Delta$ is the set of transitions and $\ell:
\Delta \to G_1\times G_2$ is the labelling function (associating a pair of group elements to each transition);
the precise definition of $\Delta$ and $\ell$ is given in what follows.

We have already seen in Proposition \ref{action} that for each $e\in E$, $\ol{e}$ induces partial transformations $\sigma_e^{(1)},
\sigma_e^{(2)}$ acting from the left on $I_1,I_2$, respectively, as well as partial transformations $\tau_e^{(1)},\tau_e^{(2)}$
acting from the right on $\Lambda_1,\Lambda_2$, respectively, with all the properties described in that proposition. So, we are
going to set
$$
\mathbf{t} = ((\lambda,i),e,(\mu,j)) \in \Delta
$$
if and only if either
\begin{itemize}
	\item $\lambda = \mu\tau_e^{(1)}$ and $\sigma_e^{(2)}i = j$, or
	\item $\lambda\tau_e^{(1)}=\mu$ and $i = \sigma_e^{(2)}j$.
\end{itemize}
If this is indeed the case, we define the label of the transition $\mathbf{t}$ by
$$
\ell(\mathbf{t}) = ((f_{i_0\lambda}^{(1)})^{-1}f_{i_0\mu}^{(1)},f_{j\lambda_0}^{(2)}(f_{i\lambda_0}^{(2)})^{-1}),
$$
where $i_0$ is an arbitrary fixed (equivalently, image) point of $\sigma_e^{(1)}$ and $\lambda_0$ is an arbitrary fixed 
(equivalently, image) point of $\tau_e^{(2)}$. This is well-defined because by Proposition \ref{action}(3) we have that the
existence  of the transition $\mathbf{t}$ guarantees that both $\sigma_e^{(1)}$ and $\tau_e^{(2)}$ are non-empty idempotent
partial maps which thus have fixed points. As already explained in the proof of that proposition, it follows from the defining
relations of groups $G_1$ and $G_2$ that the actual choice of fixed points $i_0,\lambda_0$ is irrelevant.
Note that by the very definition of $\Delta$ we have that $\mathcal{A}(D_1,D_2)$ is a
two-way automaton (which is non-deterministic, because neither $\tau_e^{(1)}$ nor $\sigma_e^{(2)}$ are injective in general); 
that is, $\Delta$ is symmetric in the sense that if $\mathbf{t} =
((\lambda,i),e,(\mu,j)) \in \Delta$ then also $\mathbf{t}^{-1} = ((\mu,j),e,(\lambda,i)) \in \Delta$. The label of
$\mathbf{t}^{-1}$ is the inverse of that of $\mathbf{t}$, namely, if $\ell(\mathbf{t})=(g_1,g_2)$ then
$$
\ell(\mathbf{t}^{-1}) = \ell(\mathbf{t})^{-1} = (g_1^{-1},g_2^{-1}).
$$
Upon omitting the labelling function $\ell$ and specifying the initial and final states, $\mathcal{A}(D_1,D_2)$ becomes a
proper non-deterministic finite automaton.

A \emph{non-empty path} in $\mathcal{A}(D_1,D_2)$ is a sequence of the form
$$
\mathbf{p} = ((\lambda_1,i_1),e_1,(\lambda_2,i_2),e_2,\dots,(\lambda_r,i_r),e_r,(\lambda_{r+1},i_{r+1}))
$$
for some $r\geq 1$ (called the \emph{length} of the path) such that for all $1\leq q\leq r$ we have $((\lambda_q,i_q),e_q,
(\lambda_{q+1},i_{q+1}))\in\Delta$. (An empty path is just an empty sequence.)
By Kleene's Theorem \cite{HUBook}, the set of words $L(\lambda,i;\mu,j)\subseteq E^\ast$
read along all the paths connecting $(\lambda,i)$ and $(\mu,j)$ is a rational language. We also attach group labels to paths,
but at this point it is important first to nuance our definition of the labelling function; $\ell$ will not, in fact,
map into the direct product $G_1\times G_2$ but rather into $G_1\times G_2^\partial$ (here $H^\partial$ denotes the
\emph{dual} of the group $H$), where the operation $\star$ in this latter direct product is defined by
$(g_1,g_2)\star(h_1,h_2) = (g_1h_1,h_2g_2)$. Now, for a path as above we define
$$
\ell(\mathbf{p}) = \ell((\lambda_1,i_1),e_1,(\lambda_2,i_2)) \star \cdots \star \ell((\lambda_r,i_r),e_r,(\lambda_{r+1},i_{r+1})).
$$
The empty (trivial) path carries the label $(1_{G_1},1_{G_2})$.

The principal reason for introducing contact automata is explained by the following result.

\begin{thm}\label{thm-1}
For $1\leq k\leq m$, let $(i_k,g_k,\lambda_k),(j_k,h_k,\mu_k)\in D_k$, where $D_k$ is a $\D$-class of $\ig{\cE}$,
such that both $$(i_1,g_1,\lambda_1),\dots,(i_m,g_m,\lambda_m)$$ and $$(j_1,h_1,\mu_1),\dots,(j_m,h_m,\mu_m)$$
form irreducible r-sequences (with $\D$-fingerprint $(D_1,\dots,D_m)$). Then
$$
(i_1,g_1,\lambda_1)\dots (i_m,g_m,\lambda_m) = (j_1,h_1,\mu_1)\dots (j_m,h_m,\mu_m)
$$
holds in $\ig{\cE}$ if and only if $i_1=j_1$, $\lambda_m=\mu_m$, and there exist:
\begin{itemize}
	\item[(a)] for all $1\leq s<m$, a path $\mathbf{p}_s$ in the contact automaton $\mathcal{A}(D_s,D_{s+1})$ from
	$(\lambda_s,i_{s+1})$ to $(\mu_s,j_{s+1})$, and
	\item[(b)] $x_t\in G_t$, $2\leq t\leq m-1$,
\end{itemize}
such that
\begin{itemize}
	\item[(i)] $\ell(\mathbf{p}_1) = (g_1^{-1}h_1, x_2)$,
	\item[(ii)] $\ell(\mathbf{p}_r) = (g_r^{-1}x_r^{-1}h_r,x_{r+1})$ for all $2\leq r\leq m-2$,
	\item[(iii)] $\ell(\mathbf{p}_{m-1}) = (g_{m-1}^{-1}x_{m-1}^{-1}h_{m-1},h_mg_m^{-1})$.
\end{itemize}
\end{thm}

\begin{proof}
We begin with the following observation.

\medskip

\noindent\emph{Claim.}
We have
$$
((k_1,a_1,\nu_1),\dots,(k_m,a_m,\nu_m)) \lra ((k'_1,b_1,\nu'_1),\dots,(k'_m,b_m,\nu'_m))
$$
for two irreducible r-sequences with $\D$-fingerprint $(D_1,\dots,D_m)$ such that
$$(k_i,a_i,\nu_i)=(k_i',b_i,\nu'_i)\ol{e}\ \text{ and }\ (k'_{i+1},b_{i+1},\nu'_{i+1})=\ol{e}(k_{i+1},a_{i+1},
\nu_{i+1})$$ for some $e\in E$ and $i<m$, and $(k_l,a_l,\nu_l)=(k'_l,b_l,\nu'_l)$ for all $l\not\in\{i,i+1\}$
(that is, we have Case (i) from the definition of $\lra$)
if and only if the contact automaton $\mathcal{A}(D_i,D_{i+1})$ contains the transition
$$
\mathbf{t} = ((\nu_i,k_{i+1}),e,(\nu'_i,k'_{i+1}))
$$
with $\ell(\mathbf{t}) = (a_i^{-1}b_i,b_{i+1}a_{i+1}^{-1})$.

\medskip

Indeed, this claim follows as a rather straightforward consequence of Proposition \ref{action} and the very
definition of $\mathcal{A}(D_i,D_{i+1})$, bearing in mind that we have $(a_i,a_{i+1})\star\ell(\mathbf{t})=
(b_i,b_{i+1})$. Also, a claim dual to the above one also holds, corresponding to case (ii) of the definition of 
the relation $\lra$.

Now, by Theorem \ref{wp-sim} and the subsequent remarks at the end of the previous section, the given equality
$$(i_1,g_1,\lambda_1)\dots (i_m,g_m,\lambda_m) = (j_1,h_1,\mu_1)\dots (j_m,h_m,\mu_m)$$
is equivalent to the existence of the following array of irreducible r-sequences:
\begin{align*}
((i_1,g_1,\lambda_1),\dots,(i_m,g_m,\lambda_m)) &= ((k_1^{(1)},a_1^{(1)},\nu_1^{(1)}),\dots,(k_m^{(1)},a_m^{(1)},\nu_m^{(1)})) \\
((k_1^{(1)},a_1^{(1)},\nu_1^{(1)}),\dots,(k_m^{(1)},a_m^{(1)},\nu_m^{(1)})) &\lra
((k_1^{(2)},a_1^{(2)},\nu_1^{(2)}),\dots,(k_m^{(2)},a_m^{(2)},\nu_m^{(2)}))\\
&\vdots \\
((k_1^{(q-1)},a_1^{(q-1)},\nu_1^{(q-1)}),\dots) &\lra ((k_1^{(q)},a_1^{(q)},\nu_1^{(q)}),\dots)\\
((k_1^{(q)},a_1^{(q)},\nu_1^{(q)}),\dots,(k_m^{(q)},a_m^{(q)},\nu_m^{(q)})) &= ((j_1,h_1,\mu_1),\dots,(j_m,h_m,\mu_m))
\end{align*}
By our Claim, such an array exists if and only if $i_1=k_1^{(1)}=\dots=k_1^{(q)}=j_1$ and $\lambda_m=\nu_m^{(1)}=\dots=\nu_m^{(q)}
=\mu_m$, and each automaton $\mathcal{A}(D_s,D_{s+1})$ ($1\leq s<m$) contains a path $\mathbf{p}_s$ (the sum of lengths of
these paths must be $q-1$) from $(\lambda_s,i_{s+1})$ to $(\mu_s,j_{s+1})$ such that if $\ell(\mathbf{p}_s) = (y_s,x_{s+1})$ (here
$x_s,y_s\in G_s$) then
\begin{align*}
g_1y_1 &= h_1,\\
x_2g_2y_2 &= h_2,\\
&\vdots \\
x_{m-1}g_{m-1}y_{m-1} &= h_{m-1},\\
x_mg_m &= h_m.
\end{align*}
Solving this system for $y_1,\dots,y_{m-1}$ and $x_m$ yields the desired result.
\end{proof}

So, let $G_1,\dots,G_m$ be finitely presented groups, while $\rho_1,\dots,\rho_{m-1}$ are rational subsets of direct products
$G_1\times G_2^\partial,\dots,G_{m-1}\times G_m^\partial$, respectively. Given these parameters, we define the following algorithmic
problem $\mathbf{P}(G_1,\dots,G_m;\rho_1,\dots,\rho_{m-1})$:

\bigskip

\noindent INPUT: $a_k,b_k\in G_k$ ($1\leq k\leq m$).\\
OUTPUT: Decide if there exist $x_t\in G_t$, $2\leq t\leq m-1$, such that
		  \begin{align*}
      (a_1^{-1}b_1, x_2) &\in \rho_1,\\
      (a_r^{-1}x_r^{-1}b_r,x_{r+1}) &\in \rho_r\quad\quad (2\leq r\leq m-2),\\
      (a_{m-1}^{-1}x_{m-1}^{-1}b_{m-1},b_ma_m^{-1}) &\in \rho_{m-1}.
      \end{align*}

\bigskip

\begin{thm}\label{prob-p}
For $1\leq k\leq m$, let $(i_k,g_k,\lambda_k),(j_k,h_k,\mu_k)\in D_k$, where $D_k$ is a $\D$-class of $\ig{\cE}$,
such that both  $$(i_1,g_1,\lambda_1),\dots,(i_m,g_m,\lambda_m)$$ and $$(j_1,h_1,\mu_1),\dots,(j_m,h_m,\mu_m)$$
form irreducible r-sequences, and let $I_k$ (resp.\ $\Lambda_k$) be an index set for the $\R$-(resp.\ $\L$-)classes of $D_k$. There exist rational subsets
$\rho_s(\lambda,i;\mu,j)$ of $G_s\times G_{s+1}^\partial$, $1\leq s<m$, $\lambda,\mu\in\Lambda_s$, $i,j\in I_{s+1}$, that are
effectively computable from the biordered set $\cE$, such that the equality
$$
(i_1,g_1,\lambda_1)\dots (i_m,g_m,\lambda_m) = (j_1,h_1,\mu_1)\dots (j_m,h_m,\mu_m)
$$
holds in $\ig{\cE}$ if and only if $i_1=j_1$, $\lambda_m=\mu_m$, and the problem
$$
\mathbf{P}(G_1,\dots,G_m;\rho_1(\lambda_1,i_2;\mu_1,j_2),\dots,\rho_{m-1}(\lambda_{m-1},i_m;\mu_{m-1},j_m))
$$
returns a positive answer on input $g_k,h_k\in G_k$, $1\leq k\leq m$.
\end{thm}

\begin{proof}
Let us start by modifying the concept of a contact automaton $\mathcal{A}(D,D')$ into a transition system $\mathcal{A}'(D,D')$
so that the new labelling function $\ell'$ maps from the set of transitions $\Delta$ into the free monoid $\Gamma^+$.
The alphabet $\Gamma$ consists of letters $a_{i\lambda},b_{i\lambda}$, $(i,\lambda)\in K\subseteq I\times\Lambda$ (where $K$ is
the set of all pairs for which $f_{i\lambda}$ is a generator of $G$, the maximal subgroup of $D$), and
$c_{i\lambda},d_{i\lambda}$, $(i,\lambda)\in K'\subseteq I'\times\Lambda'$ (with $K'$ being the set of all pairs
for which $f'_{i\lambda}$ is a generator of $G'$, the maximal subgroup of $D'$). So, for $\mathbf{t} =
((\lambda,i),e,(\mu,j)) \in \Delta$ we set
$$
\ell'(\mathbf{t}) = b_{i_0\lambda}a_{i_0\mu}d_{j\lambda_0'}c_{j\lambda_0'},
$$
where $i_0\in I$ is an arbitrary fixed point of $\sigma_e$ and $\lambda_0'\in\Lambda'$ is an arbitrary fixed point of $\tau'_e$.
The label of an empty path is the empty word. With these modifications, $\mathcal{A}'(D,D')$ is a finite state transducer from
$E^\ast$ to $\Gamma^\ast$ (see \cite{BeBook} for general background on finite state transducers and rational transductions).

Let $L_s(\lambda,i;\mu,j)\subseteq E^\ast$ denote the rational language of all words over $E$ that take the state $(\lambda,i)$
to the state $(\mu,j)$ in $\mathcal{A}'(D_s,D_{s+1})$ (or, equivalently, in $\mathcal{A}(D_s,D_{s+1})$). Furthermore, let
$R_s(\lambda,i;\mu,j)$ denote the language of all words over
$$
\Gamma_s=\{a_{r\nu}^{(s)},b_{r\nu}^{(s)},c_{r'\nu'}^{(s)},d_{r'\nu'}^{(s)}:\
r\in I^{(s)}, r'\in I^{(s+1)},\nu\in\Lambda^{(s)},\nu'\in\Lambda^{(s+1)}\}
$$
arising as $\ell'_s(\mathbf{p})$ such that
$\mathbf{p}$ is a path in $\mathcal{A}'(D_s,D_{s+1})$ starting at $(\lambda,i)$ and ending at $(\mu,j)$. Then $R_s(\lambda,i;\mu,j)$
is a rational transduction of $L_s(\lambda,i;\mu,j)$, so it is rational itself. Since the structure of both $\mathcal{A}(D,D')$ and
$\mathcal{A}'(D,D')$ depend solely on the maps $\sigma_e,\sigma'_e,\tau_e,\tau'_e$ which are already shown in Proposition \ref{action}
to be computable from $\cE$, there is an algorithm effectively constructing these automata from $\cE$, so by virtue of Kleene's Theorem
all the rational languages $L_s(\lambda,i;\mu,j)$ and $R_s(\lambda,i;\mu,j)$ are computable given $\cE$. Now let
$\psi_s:\Gamma_s^\ast\to G_s\times G_{s+1}^\partial$ be the (monoid) homomorphism uniquely extending the map
\begin{align*}
a_{r\nu}^{(s)} &\mapsto (f_{r\nu}^{(s)},1_{G_{s+1}}), \\
b_{r\nu}^{(s)} &\mapsto ((f_{r\nu}^{(s)})^{-1},1_{G_{s+1}}), \\
c_{r'\nu'}^{(s)} &\mapsto (1_{G_s},f_{r'\nu'}^{(s+1)}), \\
d_{r'\nu'}^{(s)} &\mapsto (1_{G_s},(f_{r'\nu'}^{(s+1)})^{-1}).
\end{align*}
Then $\rho_s(\lambda,i;\mu,j) = R_s(\lambda,i;\mu,j)\psi_s$ and we have that $(g,h)\in \rho_s(\lambda,i;\mu,j)$ if and only if
$\ell(\mathbf{p})=(g,h)$ for some path $\mathbf{p}$ in $\mathcal{A}(D_s,D_{s+1})$ from $(\lambda,i)$ to $(\mu,j)$. Hence,
$\rho_s(\lambda,i;\mu,j)$ is a rational subset of $G_s\times G_{s+1}^\partial$, and our required result now follows directly
from the previous theorem.
\end{proof}

\begin{rmk}
Due to the fact that $\cE$ is finite, for any $s$ and regular $\D$-classes $D_s,D_{s+1}$ there are only finitely many rational
subsets of $G_s\times G_{s+1}^\partial$ of the form $\rho_s(\lambda,i;\mu,j)$ simply because both $\Lambda_s$ and $I_{s+1}$ are
finite index sets. Therefore, given a $\D$-fingerprint $(D_1,\dots,D_m)$ of an equality $\ol{u}=\ol{v}$ to be decided in the
word problem of $\ig{\cE}$, there are only finitely many problems of the form
$$\mathbf{P}(G_1,\dots,G_m;\rho_1,\dots,\rho_{m-1})$$
involved. In each of these problems, all the parameters (the presentations of groups $G_i$ and the rational expressions denoting
$\rho_i$) are effectively computable from $\cE$ by the previous theorem and \cite[Theorem 3.10]{DGR}, so the reduction of the word
problem of $\ig{\cE}$ described in the previous theorem is effective.
\end{rmk}

\begin{rmk}\label{rem-12}
\begin{itemize}
\item[(i)] For $m=1$, deciding $\ol{u}=\ol{v}$ comes down to deciding equality of regular elements of $\ig{\cE}$ and this translates
to asking whether $(i,g,\lambda)=(j,h,\mu)$ (where $g,h$ are two elements of $G$, the maximal subgroup of $D$, represented as group
words over the generating set of $G$). Clearly, this holds if and only if $i=j$, $\lambda=\mu$ and $g=h$ holds in $G$. Therefore,
the `regular segment' of the word problem of $\ig{\cE}$ is indeed equivalent to the word problem of its maximal subgroups.
The same conclusion is reached in \cite[Theorem 3.10]{DGR}.
\item[(ii)] For $m=2$, the problem $\mathbf{P}(G_1,G_2;\rho)$ asks, on input $a_1,b_1\in G_1$, $a_2,b_2\in G_2$, whether
$(a_1^{-1}b_1,b_2a_2^{-1})\in\rho$ (note that no existentially bound variables are involved in this case); hence,
$\mathbf{P}(G_1,G_2;\rho)$ is simply the membership problem for $\rho\subseteq G_1\times G_2^\partial$. This is the basis of the
undecidablity result from \cite{DGR}. There, for a finitely presented group $G$ and its finitely generated subgroup $H$,
(a biordered set of) an idempotent semigroup $\mathcal{B}_{G,H}$ is constructed such that in a certain pair of $\D$-classes $D,D'$
we have that its maximal subgroups are both isomorphic to $G$, and for a specific choice of $i\in I$, $i'\in I'$, $\lambda\in
\Lambda$ and $\lambda'\in \Lambda'$ we have
$$
\rho=\rho(\lambda,i';\lambda,i') = \{(h^{-1},h):\ h\in H\}.
$$
It is well-known \cite{Mi} that the square of a free group of finite rank at least 2 has a finitely generated subgroup $H$ with
an undecidable membership problem. For such a choice of $G$ and $H$, we clearly have that the membership problem of $\rho$ (and
thus the problem $\mathbf{P}(G,G;\rho)$) must also be undecidable, and this is reflected in the undecidability of the word
problem of $\ig{\mathcal{B}_{G,H}}$ (see \cite[Proposition 8.1]{DGR}).
\end{itemize}
\end{rmk}

\subsection*{Summary of the reduction of the word problem of $\ig{\cE}$}

\begin{itemize}
\item[(1)] Given two words $u,v\in E^+$, compute their $\D$-fingerprints via finding a minimal r-factorisation for each of
them and then locating a seed for each of the factors, as described in Remark \ref{wp-eff}. If the $\D$-fingerprints fail
to coincide then we already know that $\ol{u}\neq\ol{v}$ in $\ig{\cE}$ by Theorem \ref{D-fing}, so we are done. Otherwise,
let $(D_1,\dots,D_m)$ be the computed joint $\D$-fingerprint of $u$ and $v$, and let $u=u_1\dots u_m$ and $v=v_1\dots v_m$
be the computed minimal r-factorisations, with $\ol{u_s},\ol{v_s}\in D_s$ for all $1\leq s\leq m$.
\item[(2)] Perform the `Rees matrix coordinatisation' of each of the $\D$-classes $D_s$ and transform the words $u_s,v_s$
into Rees matrix triples by using the formula from Theorem \ref{thm:Rees}. This turns the question of whether $\ol{u}=
\ol{u_1}\dots\ol{u_m} = \ol{v_1}\dots\ol{v_m}=\ol{v}$ holds into an equality of two products (each of length $m$) of triples
belonging to $D_1,\dots,D_m$, respectively.
\item[(3)] For each $s$, $1\leq s\leq m$, compute the partial maps $\sigma_e^{(s)},\tau_e^{(s)}$ (for any
suitable $e$) arising from the $\D$-class
$D_s$, as explained in Proposition \ref{action}. This allows us to effectively construct the contact automata $\mathcal{A}(D_s,D_{s+1})$
and thus, by applying Theorem \ref{prob-p} and the standard algorithm from the proof of Kleene's Theorem for analysis of
automata (see \cite{HUBook}), to construct a rational expression for each subset $\rho_s(\lambda,i;\mu,j)$ over
the generating set of the group $G_s\times G_{s+1}^\partial$.
\item[(4)] Finally, to decide whether $$(i_1,g_1,\lambda_1)\dots (i_m,g_m,\lambda_m) = (j_1,h_1,\mu_1)\dots (j_m,h_m,\mu_m)$$
(which is now the form in which we ask whether $\ol{u}=\ol{v}$), establish whether $i_1=j_1$ and $\lambda_m=\mu_m$ and then,
if it is recursively soluble, invoke the algorithmic problem $$\mathbf{P}(G_1,\dots,G_m;\rho_1(\lambda_1,i_2;\mu_1,j_2),\dots,
\rho_{m-1}(\lambda_{m-1},i_m;\mu_{m-1},j_m)).$$
By Theorem \ref{prob-p} the answer to the above problem determines whether $\ol{u}=\ol{v}$.
\end{itemize}

We have now completed all the stages of this reduction. In the next, final section we are going to present an application of
this reduction by proving that an important class of examples of biorders yields free idempotent generated semigroups $\ig{\cE}$
with a decidable word problem. These will be biordered sets $\cE$ such that all the non-maximal $\D$-classes of $\ig{\cE}$ have
finite maximal subgroups.

\section{Application I: Some soluble word problems}
\label{sec:appl1}

Let $S$ be a semigroup with finitely many idempotents. If $S$ has an identity element $1$ (so that it is a monoid) then its
$\J$-class $J_1$ contains a single $\D$-class and the corresponding principal factor $J_1^0$ is completely 0-simple; it follows
that $1$ is the only idempotent in $J_1$, and the latter is consequently the group of units of $S$. Furthermore, $J_1$ is the
unique maximal $\J$-class of $S$. If $S$ is, in addition, idempotent generated, then $J_1=D_1=\{1\}$.

We are going to say that a regular $\D$-class $D$ of $\ig{\cE}$ is \emph{maximal} if there is no regular $\D$-class $D'$
such that $D < D'$ in the $\J$-order, unless $\cE$ (and so $\ig{\cE}$) has an identity element $1$ when we allow
$D<\{1\}$, but there is no regular $\D$-class $D'$ such that $D < D' < \{1\}$. As remarked just after the relations \eqref{sing-sq}
and the definition of singular squares, if there are no singular squares in $D$ then the maximal subgroup(s) of $D$ must be
either free or trivial (which may, of course, be considered as a free group of rank 0). This is certainly the case with maximal $\D$-classes,
as the identity element 1 cannot singularise any square. The main result we aim to prove in this section is as follows.

\begin{thm}\label{thm:appl}
Let $\cE$ be a finite biordered set with the property that the maximal subgroups in all non-maximal $\D$-classes of $\ig{\cE}$
are finite. Then $\ig{\cE}$ has a decidable word problem.
\end{thm}

For this goal, we need a bit more preparatory work.

\begin{lem}\label{II-L1}
Let $D$ be a regular $\D$-class of $\ig{\cE}$, and let $\mathbf{x},\mathbf{y}\in D$
be such that $\mathbf{x}=(i,g,\lambda)$ and $\mathbf{y}=(j,h,\mu)$. Then $\mathbf{xy}$ is
a regular element of $\ig{\cE}$ if and only if $p_{\lambda j}\neq 0$ (i.e.\ if and only if $L_{\mathbf{x}}\cap R_{\mathbf{y}}$
contains an idempotent).
\end{lem}

\begin{proof}
If $p_{\lambda j}\neq 0$ then we simply have
$$
\mathbf{xy} = (i,g,\lambda)(j,h,\mu) = (i,gp_{j\lambda}h,\mu) = (i,gf_{j\lambda}^{-1}h,\mu)\in D,
$$
showing that $\mathbf{xy}$ is regular.

Conversely, assume that $p_{\lambda j}=0$, so that $L_{\mathbf{x}}\cap R_{\mathbf{y}}$ contains no idempotent. Then it
immediately follows that $\mathbf{xy}\not\in D$, since in the corresponding principal factor $D^0$ we have $\mathbf{xy}=0$.
More precisely, we have that $\mathbf{xy}\in D'$ for some $\D$-class $D'$ that is $\J$-below $D$; so, if we assume (seeking
a contradiction) that $\mathbf{xy}$ is a regular element of $\ig{\cE}$, it follows that $D'$ is a regular $\D$-class.
Hence, by Theorem \ref{thm:DGR} it follows that $\mathbf{xy}=\ol{uev}$ for some $e\in E$ and $u,v\in E^\ast$ such that
$\ol{e}\in D'$. On the other hand, by Lemma \ref{lem:FG} both $\mathbf{x}$ and $\mathbf{y}$ can be represented
as products of idempotents from $D$: $\mathbf{x}=\ol{w_1}$ and $\mathbf{y}=\ol{w_2}$ for some $w_1,w_2\in E_D^+$. Thus
we get that $\ol{w_1w_2}=\ol{uev}$; in particular, there are words $u_0,u_1,\dots,u_n$ such that
$$
w_1w_2=u_0 \equiv u_1\equiv\dots\equiv u_n=uev.
$$
However, this is not possible, as each rewriting rule stemming from the presentation of $\ig{\cE}$ either replaces a
two-letter word $ef$ consisting of letters corresponding to idempotents $\ol{e},\ol{f}$ from $\J$-comparable $\D$-classes
by a letter $g$ such that $\ol{g}$ is $\D$-related to one of the previous idempotents from the `lower' $\D$-class, or
replaces a letter $g$ with a word $ef$ such that one of $\ol{e},\ol{f}$ is $\D$-related to $\ol{g}$ and the other is
from a $\D$-class that is $\J$-above $D_{\ol{g}}$. Hence, starting from $u_0\in E_D^+$ one can never obtain a word
that contains a letter $e$ such that $\ol{e}\in D'$, a contradiction.

(By the way, this shows that $w_1w_2$ is a minimal r-factorisation, so that the $\D$-fingerprint of $\mathbf{xy}$
is $(D,D)$.)
\end{proof}

\begin{lem}\label{II-L2}
Let $D,D'$ be distinct maximal regular $\D$-classes of $\ig{\cE}$, and let $G,G'$ be maximal subgroups of $D$ and $D'$, 
respectively.
\begin{itemize}
\item[(i)] The contact automaton $\mathcal{A}(D,D')$ is either empty or contains only loops around each vertex labelled
by $(1_G,1_{G'})$ if $\cE$ has an identity element $1$.
\item[(ii)] The restriction of the contact automaton $\mathcal{A}(D,D)$ (of $D$ with itself) to the vertices $(\lambda,j)$
such that $p_{\lambda j}=0$ is either empty or contains only loops around each vertex labelled by $(1_G,1_G)$ if $\cE$ has an
identity element $1$.
\end{itemize}
\end{lem}

\begin{proof}
(i) Assume that $\mathcal{A}(D,D')$ contains a transition $((\lambda_1,j_1'),e,(\lambda_2,j_2'))$ such that $e\neq 1$, $\lambda_1,
\lambda_2\in \Lambda$ and $j_1',j_2'\in I'$. Then, in particular, we must have either $(i,g,\lambda_1)\ol{e}=(i,h,\lambda_2)$, or
$(i,g,\lambda_1)=(i,h,\lambda_2)\ol{e}$ for some $i\in I$, $g,h\in G$; either way, $D\leq D_{\ol{e}}$. Analogously, we arrive
at $D'\leq D_{\ol{e}}$, which is impossible since both $D,D'$ are maximal and $e$ is non-identity, a contradiction.

(ii) Assume that $p_{\lambda j}=0$ while there is a transition $((\lambda,j),e,(\lambda',j'))$ in the automaton $\mathcal{A}(D,D)$
for some pair $(\lambda',j')$ such that $e\neq 1$. Then we must have $\ol{e}\in D$ (as $D\leq D_{\ol{e}}$ and $D$ is maximal),
and either $(i,g,\lambda)=(i,g',\lambda')\ol{e}$ and $\ol{e}(j,h,\mu)=(j',h',\mu)$ for some $i,\mu$ and group elements $g,g',h,h'\in G$,
or, conversely, $(i,g,\lambda)\ol{e}=(i,g',\lambda')$ and $(j,h,\mu)=\ol{e}(j',h',\mu)$. This means that if
$$
\ol{e} = (k,f_{k\nu},\nu)
$$
we have $\nu=\lambda$ in the first case and $k=j$ in the second; in either case we reach a contradiction because of $p_{\lambda j}=0$.
\end{proof}

\begin{lem}\label{II-L3}
Let $G$ be a finite group and $F$ a free group of finite rank. If $\rho$ is a rational subset of
$G\times F^\partial$ then for each $g\in G$, the set
$$
g\rho = \{w\in F:\ (g,w)\in\rho\}
$$
is a rational subset of $F^\partial$ (and of $F$ as well).
\end{lem}

\begin{proof}
First of all, any group $H$ is isomorphic to its dual $H^\partial$. So, once we prove that $g\rho$ is a rational
subset of $F^\partial$, the analogous assertion about $F$ follows immediately.

Also, $F^\partial$ is now a free group, so $G\times F^\partial$ is a virtually free group. Therefore,
we can use the result of Grunschlag \cite{Gr} (see also \cite[Subsection 3.4]{BS} and \cite{Si}) describing rational subsets
of finitely generated virtually free groups. Namely, if we start with such a virtually free group and
consider its normal free subgroup of finite rank and finite index -- in our case, $\{1_G\}\times F^\partial$ --
and fix a set of right coset representatives -- which is $\{(g,1_F):\ g\in G\}$ in our case -- then any
rational subset $\rho$ of $G\times F^\partial$ can be written as
$$
\rho = \bigcup_{g\in G} R_g(g,1_F)
$$
for some rational subsets $R_g$ of $\{1_G\}\times F^\partial$. Since $g\rho$ is simply the image of $R_g$
under the natural isomorphism $\{1_G\}\times F^\partial \to F^\partial$, the lemma follows.
\end{proof}

A dual result also holds for rational subsets $\rho$ of direct products $F\times G^\partial$ where $F$ is free
and $G$ is finite. For a rational subset $\rho$ of $F\times G$ we shall denote by $\rho g$ the set
$\{w\in F:\ (w,g)\in\rho\}$ (which now also follows to be a rational subset of $F$ and $F^\partial$).
It is not difficult to extract from the proof of Proposition 4.1 in \cite{Si} that given $g\in G$ and a 
rational subset $\rho$ of $G\times F^\partial$ (resp.\ $F\times G^\partial$), the rational subset $g\rho$
(resp.\ $\rho g$) of $F$ is effectively computable.

\begin{proof}[Proof of Theorem \ref{thm:appl}]
Following steps (1)--(4) from the summary from the previous section, our task reduces to deciding equalities
of the form
\begin{equation}\label{eqdec}
(i_1,g_1,\lambda_1)\dots (i_m,g_m,\lambda_m) = (j_1,h_1,\mu_1)\dots (j_m,h_m,\mu_m)
\end{equation}
in $\ig{\cE}$, where the elements in the above products form irreducible r-sequences, both with $\D$-fingerprint
$(D_1,\dots,D_m)$. Note that if $\cE$ contains an identity element, we can safely assume that it does not appear
in the above equality. By Theorem \ref{prob-p}, it suffices to  check whether $i_1=j_1$ and $\lambda_m=\mu_m$,
and then invoke the problem
$$
\mathbf{P}(G_1,\dots,G_m;\rho_1(\lambda_1,i_2;\mu_1,j_2),\dots,\rho_{m-1}(\lambda_{m-1},i_m;\mu_{m-1},j_m)),
$$
for input $g_k,h_k\in G_k$, where $G_k$ is the maximal subgroup of $D_k$, $1\leq k\leq m$. Note that these groups
are either finite or free (of finite rank); the latter happens only if $D_k$ is a maximal $\D$-class of $\ig{\cE}$.

Note that Lemmas \ref{II-L1} and \ref{II-L2}(ii) imply that whenever $D_k=D_{k+1}=D$ is a maximal $\D$-class (with
a free maximal subgroup $G_k=G_{k+1}$), we have $p_{\lambda_k i_{k+1}}=p_{\mu_k j_{k+1}}=0$ in the
principal factor of $D$ and so
$$
\rho_k(\lambda_k,i_{k+1};\mu_k,j_{k+1}) = \left\{
\begin{array}{cl}
\{(1_{G_k},1_{G_{k+1}})\} & \text{if }(\lambda_k,i_{k+1})=(\mu_k,j_{k+1}),\\[1.5mm]
\es & \text{otherwise.}
\end{array}
\right.
$$
Moreover, by Lemma \ref{II-L2}(i) it immediately follows that the same conclusion holds when $D_k$ and $D_{k+1}$
are two distinct maximal $\D$-classes.

Consider the set $X\subseteq [1,m]$ of all indices $k$ such that $G_k$ is \emph{not} finite, i.e.\
such that $D_k$ is a maximal $\D$-class with a nontrivial free maximal subgroup $G_k$. 
So, for \eqref{eqdec} to hold, we must have $\lambda_k=\mu_k$ and $i_{k+1}=j_{k+1}$ whenever $k,k+1\in X$, for otherwise
$$\mathbf{P}(G_1,\dots,G_m;\rho_1(\lambda_1,i_2;\mu_1,j_2),\dots,\rho_{m-1}(\lambda_{m-1},i_m;\mu_{m-1},j_m))$$
would fail, as it would contain at least one empty relation. Hence, our theorem will be proved
as soon as we prove that the problem
$\mathbf{P}(G_1,\dots,G_m;\rho_1,\dots,\rho_{m-1})$ is decidable whenever $\rho_k=\{(1_{G_k},1_{G_{k+1}})\}$
for all $k$ such that $k,k+1\in X$. To keep in line with the original definition of this problem
(and step away from the particular situation in which it is invoked), we assume that the input consists of 
elements $a_k,b_k\in G_k$, $1\leq k\leq m$.

Consider first the case $m=1$. As already noted in Remark \ref{rem-12}(i), the problem we are considering boils down
to the word problem of the maximal subgroup $G$ of a single $\D$-class. As $G$ is either free or finite, it has a decidable
word problem.

Now let $m=2$, so that the $\D$-fingerprint we are considering consists of a pair of $\D$-classes $(D_1,D_2)$. As noted in
Remark \ref{rem-12}(ii), we are faced with the task of deciding whether $(a^{-1}b_1,b_2a_2^{-1})\in\rho_1$. If both
$D_1,D_2$ are maximal, then $\rho_1=\{(1_{G_1},1_{G_2})\}$, so our problem is equivalent to $a_1=b_1$ and $a_2=b_2$.
On the other hand, if both $D_1,D_2$ are not maximal, then their maximal subgroups are finite, and so is the relation
$\rho_1$. Thus our problem is again decidable. Finally assume that $D_1$ is maximal, while $D_2$ is not (the converse
case is analogous). Then $G_1$ is free and $G_2$ is finite. The condition $(a^{-1}b_1,b_2a_2^{-1})\in\rho_1$ can be
rewritten as $a^{-1}b_1\in\rho_1(b_2a_2^{-1})$. By the remarks following Lemma \ref{II-L3}, $\rho_1(b_2a_2^{-1})$ is 
an effectively computable rational subset of the free group $G_1$, so the latter question is decidable by Benois' Theorem.

Therefore, in the remainder of the proof we may assume that $m\geq 3$. Let $x_s\mapsto \xi_s$ be an arbitrary assignment 
of values from the finite groups $G_s$ to variables $x_s$ such that $s\not\in X$; we will call this simply an \emph{assignment}. 
Clearly, there are only finitely many assignments. Define the problem
$$\mathbf{P}(G_1,\dots,G_m;\rho_1,\dots,\rho_{m-1};\xi_s)_{s\not\in X}$$ 
to be the original decision problem $\mathbf{P}(G_1,\dots,G_m;\rho_1,\dots,\rho_{m-1})$ but with the value of each
variable $x_s$ such that $s\not\in X$ fixed at $\xi_s$. It is now immediate that $\mathbf{P}(G_1,\dots,G_m;\rho_1,\dots,\rho_{m-1})$
yields a positive answer if and only if at least one of the finitely many problems $\mathbf{P}(G_1,\dots,G_m;\rho_1,\dots,\rho_{m-1};
\xi_s)_{s\not\in X}$ yields a positive answer. Our aim is thus to show that each of the latter problems is decidable.

To this end, for each $k\in X$, $2\leq k\leq m-1$, we examine the conditions imposed on the variable $x_k$. If $x_k$ appears within
the second coordinate of a condition from $\mathbf{P}(G_1,\dots,G_m;\rho_1,\dots,\rho_{m-1};\xi_s)_{s\not\in X}$, the possibilities 
are as follows:
\begin{itemize}
\item $(a_1^{-1}b_1,x_2)\in\rho_1$ if $k=2$;
\item $(a_{k-1}^{-1}x_{k-1}^{-1}b_{k-1},x_k)\in \rho_{k-1}$ if $k>2$ and $k-1\in X$;
\item $(a_{k-1}^{-1}\xi_{k-1}^{-1}b_{k-1},x_k)\in \rho_{k-1}$ if $k>2$ and $k-1\not\in X$.
\end{itemize}
These conditions yield, respectively, that 
\begin{itemize}
\item $a_1=b_1$ and $x_2=1_{G_2}$ if $1\in X$, and otherwise $x_2\in(a_1^{-1}b_1)\rho_1$;
\item $x_k=1_{G_k}$;
\item $x_k\in (a_{k-1}^{-1}\xi_{k-1}^{-1}b_{k-1})\rho_{k-1}$.
\end{itemize}
On the other hand, for conditions where $x_k$ appears within the first coordinate we have the following possibilities:
\begin{itemize}
\item $(a_k^{-1}x_k^{-1}b_k,x_{k+1})\in \rho_k$ if $k+1<m$ and $k+1\in X$;
\item $(a_k^{-1}x_k^{-1}b_k,\xi_{k+1})\in \rho_k$ if $k+1<m$ and $k+1\not\in X$;
\item $(a_k^{-1}x_k^{-1}b_k,b_ma_m^{-1})\in\rho_k$ if $k+1=m$.
\end{itemize}
These conditions yield, respectively, that 
\begin{itemize}
\item $x_k=b_ka_k^{-1}$;
\item $x_k\in b_k(\rho_k\xi_{k+1})^{-1}a_k^{-1}$;
\item $x_k=b_ka_k^{-1}$ and $a_m=b_m$ if $m\in X$, and otherwise $$x_k\in b_k(\rho_k(b_ma_m^{-1}))^{-1}a_k^{-1}.$$
\end{itemize}

In summary, a problem of the form $\mathbf{P}(G_1,\dots,G_m;\rho_1,\dots,\rho_{m-1};\xi_s)_{s\not\in X}$ is 
equivalent to checking a certain collection of conditions of the following four types (where for convenience we set
$\xi_1=1$ and $\xi_m=b_ma_m^{-1}$):
\begin{itemize}
\item $a_k=b_k$ for certain values of $k$ (determined by the set $X$), which is trivially decidable (since $G_k$ is either finite or free);
\item $b_ka_k^{-1}\in (a_{k-1}^{-1}\xi_{k-1}^{-1}b_{k-1})\rho_{k-1}$ for certain values of $k$, where the right-hand side is an 
effectively computable rational subset of the free group $G_k$, and thus conditions of this type are decidable by Benois' Theorem;
\item $a_k^{-1}b_k\in\rho_k\xi_{k+1}$ for certain values of $k$ -- again, the right-hand side is an 
effectively computable rational subset of the free group $G_k$, with the same decidability conclusion as above;
\item a question whether an intersection of two computable rational subsets of a free group (which itself is rational and 
effectively computable by \cite[Corollary 3.4(i)]{BS}) is empty, which by Benois' Theorem essentially reduces to the 
(decidable) emptiness problem for rational languages.
\end{itemize}
Therefore, we conclude that, under the given conditions, each of the problems $\mathbf{P}(G_1,\dots,G_m;\rho_1,\dots,\rho_{m-1};
\xi_s)_{s\not\in X}$ is decidable.
This completes the proof of the theorem.
\end{proof}

The biordered set of idempotents of $\mathcal{T}_n$, the monoid of all transformations of an $n$-element
set, arises from its idempotent generated subsemigroup $\langle E(\mathcal{T}_n)\rangle=(\mathcal{T}_n\setminus\mathcal{S}_n)
\cup\{\mathrm{id}_n\}$. These are both regular monoids, and the structure of their (necessarily regular)
$\D$-classes is well known: they form a chain $D_n>D_{n-1}>D_{n-2}> \dots >D_1$ so that $D_r$ consists precisely
of all transformations of rank $r$ (where rank of a self-map of $n$ is the size of its image). Hence the regular
$\D$-classes of $\ig{\cE_{\mathcal{T}_n}}$ also form a chain of length $n$, and the main result of \cite{GR2}
(as well the subsequent discussion in the final section of that paper) provides the information about the
maximal subgroups. If $\ol{D_r}$ denotes the $\D$-class of $\ig{\cE_{\mathcal{T}_n}}$ corresponding to
$D_r$ then it is, naturally, trivial for $r=n$, for $r=n-1$ the maximal subgroup of $\ol{D_{n-1}}$ is the
free group of rank $\binom{n}{2}-1$, while for $r\leq n-2$ the maximal subgroup of $\ol{D_r}$ is the symmetric
group $\mathcal{S}_r$, just as in $\mathcal{T}_n$.

Analogous statements for $\mathcal{PT}_n$, the monoid of all partial transformations of an $n$-element
set, follow from the results of the paper \cite{Do2} (only the rank of the free group arising from the $\D$-class
of rank $n-1$ partial maps will be different). Therefore, the biorders of both $\mathcal{T}_n$ and $\mathcal{PT}_n$
satisfy the assumptions of Theorem \ref{thm:appl}, allowing us to deduce the following conclusion.

\begin{cor}\label{TnPTn}
For any $n\geq 1$, the free idempotent generated semigroups $\ig{\cE_{\mathcal{T}_n}}$ and
$\ig{\cE_{\mathcal{PT}_n}}$ have decidable word problems.
\end{cor}

Since the $\D$-class structure of matrix monoids $M_n(Q)$ over a division ring $Q$ (and its idempotent generated
subsemigroup $\langle E(M_n(Q))\rangle=(M_n(Q)\setminus GL_n(Q))\cup\{I_n\}$) is also well known -- the relation
$\D=\J$ simply classifies matrices according to their rank -- it would be interesting to determine whether the
biordered set $\cE_{M_n(Q)}$ falls under the scope of Theorem \ref{thm:appl}, provided $Q$ is a finite field.
We recall that it was proved in \cite{DG} that if $r<n/3$ then the maximal subgroup of $\ig{\cE_{M_n(Q)}}$
contained in its $\D$-class $\ol{D_r}$ corresponding to rank $r$ matrices is $GL_r(Q)$. Meanwhile, computational
evidence has arisen \cite{OBr} that the bound $r<n/3$ might be sharp, so that the main result of \cite{DG}
is no longer true in higher rank $\D$-classes. Therefore, it seems sensible to ask the following question.

\begin{prb}
Let $Q$ be a finite field. Is the maximal subgroup of $\ig{\cE_{M_n(Q)}}$ contained in its $\D$-class $\ol{D_r}$
(corresponding to matrices of rank $r$) finite whenever $r\leq n-2$ ?
\end{prb}

\section{Application II: The weakly abundant property}
\label{sec:appl2}

There is a particularly neat way of expressing regularity in semigroups: namely, an element $a$ of a semigroup $S$ is
regular if and only if there exist idempotents $e,f\in E(S)$ such that $e\;\L\;a\;\R\;f$ (see \cite[Proposition 2.3.2]{HoBook}). 
In other words, $S$ is regular if and only if each $\R$-class and each $\L$-class of $S$ contains an idempotent.
By imposing an analogous condition for generalisations of Green's relations $\R$ and $\L$, namely $\R^*$ and $\L^*$, and also
$\tR$ and $\tL$, we arrive at notions of abundant and weakly abundant semigroups, respectively.

Now we briefly recall some basic definitions and results from \cite{ElQ, Fountain, Lawson}. Let $S$ be a semigroup. The
relations $\R^*$ and $\L^*$ are defined on $S$ by
$$
a\;\R^*\;b ~\Leftrightarrow~(\forall x,y\in S^1)~(xa=ya\Leftrightarrow xb=yb)
$$
and 
$$
a\;\L^*\;b ~\Leftrightarrow~(\forall x,y\in S^1)~(ax=ay\Leftrightarrow bx=by)
$$
for any $a,b\in S$. As remarked in \cite{Fountain}, it is easy to see that $\R\subseteq \R^*$ and $\L\subseteq \L^*$ in any semigroup, 
and that we have $\R= \R^*$ and $\L=\L^*$ whenever $S$ is regular. Furthermore, we denote by $\H^*$ the intersection $\R^*\cap\L^*$, 
and by $\D^*$ the join $\R^*\vee\L^*$. Note that, unlike for ordinary Green's relations, in general we have $\R^*\circ \L^*\neq\L^*\circ\R^*$ 
(see \cite[Example 1.11]{Fountain}). A semigroup $S$ is \emph{abundant} if each $\L^*$-class and each $\R^*$-class contains an idempotent. 
The role that the relations $\R^*$, $\L^*$, $\H^*$ and $\D^*$ play in the theory of abundant semigroups is analogous to that of 
Green's relations in the theory of regular semigroups.

\begin{lem}\cite{Fountain}
Let $S$ be a semigroup with $a\in S$ and $e\in E(S)$. The following statements are equivalent:
\begin{itemize}
\item[(i)] $a\;\R^*\;e$;
\item[(ii)] $ea=a$ and for any $x,y\in S^1$, $xa=ya$ implies $xe=ye$.
\end{itemize}
\end{lem}

There is yet another extension of Green's relations (and, in fact, of their starred counterparts), introduced in \cite{ElQ} and further
studied in \cite{Lawson}, that is useful for non-abundant semigroups. Define relations $\tR$ and $\tL$ on a semigroup $S$ by
$$
a\;\tR\;b ~\Leftrightarrow~(\forall e\in E(S))~(ea=a\Leftrightarrow eb=b)
$$
and 
$$
a\;\tL\;b ~\Leftrightarrow~ (\forall e\in E(S))~(ae=a\Leftrightarrow be=b)
$$
for any $a,b\in S$. Clearly, $\R^*\subseteq \tR$ and $\L^*\subseteq \tL$. If $S$ is abundant, then $\R^*=\tR$ and $\L^*=\tL$
(see, for example, \cite[Theorem 1.5]{Lawson}). While $\R^*$ is always a left congruence and $\L^*$ is always a right congruence on
any semigroup $S$, this is not necessarily true for $\tR$ and $\tL$ (see \cite[Example 3.6]{Lawson}). A semigroup $S$ is 
\emph{weakly abundant} (or, following \cite{Ma,MS}, a \emph{Fountain semigroup}) if each $\tR$-class and each $\tL$-class contains 
an idempotent.  We say that a weakly abundant semigroup $S$ \emph{satisfies the congruence condition} if $\tR$ is a left congruence 
and $\tL$ is a right congruence. So, any abundant semigroup is weakly abundant with the congruence condition.

\begin{lem}\cite{Lawson}\label{MVL}
Let $S$ be a semigroup with $a\in S$ and $e\in E(S)$. The following statements are equivalent:
\begin{itemize}
\item[(i)] $a\;\tR\;e$;
\item[(ii)] $ea=a$ and for any $f\in E(S)$, $fa=a$ implies $fe=e$.
\end{itemize}
\end{lem}

As yet another application of the main results of this work, we show that for any finite biordered set $\cE$, $\ig{\cE}$ is a 
Fountain semigroup with the congruence condition. To this end, we need the following notion. A word $w$ with a minimal 
r-factori\-sation $w=p_1\dots p_m$ is said to be in \emph{reduced form} if all letters of each factor are seeds. Notice 
that every word from $E^+$ is equivalent in $\ig{\cE}$ to a word in reduced form. To see this, let $w\in E^+$. We know that 
$w$ has a minimal r-factorisation $p_1\dots p_m$ such that all $\ol{p_i}$, $1\leq i\leq m$, are regular and no nontrivial 
product of consecutive $\ol{p_i}$'s is regular. Now it follows from Lemma \ref{lem:FG} that $\ol{p_i}=\ol{e_{i1}\dots e_{il(i)}}$ 
holds for some $e_{ij}\in E$ such that $\ol{e_{ij}}\in D_{\ol{p_i}}$ for all $1\leq j\leq l(i)$; thus we obtain a word $w'$ 
in reduced form such that $\ol{w}=\ol{w'}$.

\begin{thm}\label{thm:Fountain}
Let $\cE$ be a finite biordered set. Then the free idempotent generated semigroup $\ig{\cE}$ is a Fountain semigroup 
satisfying the congruence condition.
\end{thm}

\begin{rmk}
Throughout the following proof, we will repeatedly use the following argument: if $u,v\in E^+$ are such that $\ol{uv}$ is regular
in $\ig{\cE}$, with a seed $e$ lying in $u$, then $\ol{u}$ is regular, $e$ is a seed for $u$, and $\ol{u}\;\R\;\ol{uv}$. Indeed,
if $u=u_1eu_2$ so that $\ol{u_1e}\;\L\;\ol{e}\;\R\;\ol{eu_2v}$, then also $\ol{eu_2}\;\R\;\ol{e}$, and now by Remark \ref{the-rem}
we have that $\ol{u}$ is regular with $e$ being a seed for $u$. Furthermore, since $\R$ is a left congruence, $\ol{u}=\ol{u_1eu_2}
\;\R\;\ol{u_1eu_2v}=\ol{uv}$. Dually, if $uv$ has a seed $f$ lying in $v$, then $\ol{v}$ is regular (with $f$ being a seed for $v$)
and we have $\ol{v}\;\L\;\ol{uv}$.
\end{rmk}

\begin{proof}[Proof of Theorem \ref{thm:Fountain}]
Let $w\in E^+$. By the remarks preceding the statement of the theorem, there is no loss of generality in assuming that $w$ is already 
in reduced form, so that we have a minimal r-factorisation $w=p_1\dots p_m$ such that each letter of each factor $p_i$, $1\leq i\leq m$, 
is a seed for $p_i$. Hence, we can write $p_i=e_{i1}\dots e_{il(i)}$ so that for all $1\leq j\leq l(i)$ we have $\ol{e_{ij}}\in D_{\ol{p_i}}$.

Let $e=e_{11}$. Clearly, $\ol{e}\,\ol{w}=\ol{ew}=\ol{w}$. Now assume that $f\in E$ is such that $\ol{f}\,\ol{w}=\ol{fw}=\ol{w}$. 
By repeated applications of Lemma \ref{lem1}, we conclude that $fw$ has a minimal r-factorisation $fw = p_1'\dots p_m'$ such that 
each $\ol{p_i'}$ is regular, no nontrivial product of consecutive $\ol{p_i'}$'s is regular, and $\ol{p_i'}\;\D\;\ol{p_i}$ for all 
$1\leq i\leq m$. But $p_1'=fp$ for some (possibly empty) word $p$, so we have, by Theorem \ref{D-fing}, and since $e$ is a seed for $p_1$,
$$
\ol{fp} = \ol{p_1'} \;\R\; \ol{p_1} \;\R\; \ol{e},
$$
from which it follows that $\ol{fe}=\ol{e}$. Hence, $\ol{w}\;\tR\;\ol{e}$. Dually, 
$\ol{w}\;\tL\;\ol{e_{ml(m)}}$, so that $\ig{\cE}$ is a weakly abundant (i.e.\ Fountain) semigroup, as required.

Next we show that $\ig{\cE}$ satisfies the congruence condition. To this end, assume that $w_1,w_2,z\in E^+$ are (already) in 
reduced form, and that $\ol{w_1}\;\tR\;\ol{w_2}$. Let $w_1=p_1\dots p_m$, $w_2=p_1'\dots p_s'$ and $z=q_1\dots q_k$ be the
corresponding reduced minimal r-factorisations, so that every letter is a seed for its corresponding factor. Recall that 
we have already proved that $\ol{w_1}\;\tR\;\ol{e_1}$, where $e_1$ is the first letter of $p_1$, and similarly, 
$\ol{w_2}\;\tR\;\ol{e_2}$, where $e_2$ is the first letter of $p_1'$. Therefore, $\ol{e_1}\;\tR\;\ol{e_2}$, and since 
$\ol{e_1},\ol{e_2}$ are idempotents, Lemma \ref{MVL} easily implies that in fact $\ol{e_1}\;\R\;\ol{e_2}$. Since $\R$ 
is always a left congruence, we have $\ol{ze_1}\;\R\;\ol{ze_2}$. So, if we prove that $\ol{zw_1}\;\tR\;\ol{ze_1}$ holds 
(while also proving $\ol{zw_2}\;\tR\;\ol{ze_2}$ by way of analogy), we would obtain 
$\ol{zw_1}\;\tR\;\ol{ze_1}\;\R\;\ol{ze_2}\;\tR\;\ol{zw_2}$ and thus achieve our goal of showing $\ol{zw_1}\;\tR\;\ol{zw_2}$.

Consider now the word $zw_1=q_1\dots q_kp_1\dots p_m$, and assume that for some $i,j$ we have that 
$\ol{q_i\dots q_kp_1\dots p_j}$ is regular. Then the word $q_i\dots q_kp_1\dots p_j$ contains at least one seed $g$. 
If this seed lies within $q_l$ for some $i\leq l\leq k$, then 
$$D_{\ol{q_l}}=D_{\ol{g}}=D_{\ol{q_i\dots q_kp_1\dots p_j}} \leq D_{\ol{q_i\dots q_k}} \leq D_{\ol{q_l}}$$
(where the first equality follows from Theorem \ref{thm:DGR} and Remark \ref{the-rem}). Therefore, $\ol{q_i\dots q_k}\;\D\;\ol{q_l}$, 
implying that $\ol{q_i\dots q_k}$ is regular, which is possible only if $i=l=k$. Similarly, if $g$ lies within $p_r$ for some 
$1\leq r\leq j$, then we must have $1=r=j$. It follows that products of consecutive factors within $q_1\dots q_kp_1\dots p_m$ 
that represent regular elements of $\ig{\cE}$ can only be either of the form $q_kp_1\dots p_j$, or of the form $q_i\dots q_kp_1$. 
In other words, $zw_1$ has a minimal r-factorisation of one of the forms:
\begin{itemize}
\item $q_1\dots q_kp_1\dots p_m$,
\item $q_1\dots q_{k-1}(q_kp_1\dots p_j)p_{j+1}\dots p_m$ (with a seed for $q_kp_1\dots p_j$ lying within $q_k$), or 
\item $q_1\dots q_{i-1}(q_i\dots  q_kp_1)p_2\dots p_m$ (with a seed for $q_i\dots  q_kp_1$ lying within $p_1$),
\end{itemize} 
and these can be turned into reduced form by applying Lemma \ref{lem:FG} to the elements represented by subwords in parentheses.

An analogous statement can be formulated for the word $ze_1$ with $m=1$ and $e_1$ in the role of $p_1$. Bearing this in mind,
we have three cases to consider.

\emph{Case (1):} $ze_1 = q_1\dots q_k e_1$ is a minimal r-factorisation, as written. In particular, $\ol{q_ke_1}$ is not regular.
We claim that 
$$zw_1=q_1\dots q_kp_1\dots p_m$$ 
is also a minimal r-factorisation, as written. Indeed, if $\ol{q_kp_1\dots p_j}$ is regular for some $j\leq m$, with a seed 
lying within $q_k$, then (i) of Remark \ref{the-rem} implies 
$$\ol{q_k}\;\R\;\ol{q_kp_1\dots p_j}\;\R\;\ol{q_ke_1},$$ 
contradicting the non-regularity of $\ol{q_ke_1}$. On the other hand, if $\ol{q_i\dots  q_kp_1}$ is regular 
for some $i\geq 1$, with a seed lying within $p_1$, then (again by (i) of Remark \ref{the-rem}) we have 
$$\ol{p_1}\;\L\;\ol{q_i\dots  q_kp_1}\;\L\;\ol{q_kp_1},$$ 
implying the regularity of $\ol{q_kp_1}$. However, since $e_1$ is a seed for $p_1$, $\ol{e_1}\;\R\;\ol{p_1}$,
and so $\ol{q_ke_1}\;\R\;\ol{q_kp_1}$, thus $\ol{q_ke_1}$ is regular, a contradiction. Since now both $zw_1$ and $ze_1$ have
reduced form minimal r-factorisations whose first factors are $q_1$, the facts already proved imply that 
$\ol{zw_1}\;\tR\;\ol{f}\;\tR\;\ol{ze_1}$, where $f$ is the first letter of $q_1$.

\emph{Case (2):} $ze_1 = q_1\dots q_{i-1}(q_i\dots q_k e_1)$ is a minimal r-factorisation for some $i>1$, whose reduced form is 
obtained by applying Lemma \ref{lem:FG} to the `parenthesised' suffix $q_i\dots q_k e_1$. We have that either $e_1$ is a seed for 
this suffix factor, or $i=k$ and a seed for $q_ke_1$ lies within $q_k$. Since $\ol{e_1}\;\R\;\ol{p_1}$ and $\R$ is a left congruence,
we conclude 
$$\ol{q_i\dots q_ke_1}\;\R\;\ol{q_i\dots q_kp_1},$$
which implies that $\ol{q_i\dots q_kp_1}$ must be regular. We claim that 
$$zw_1 = q_1\dots q_{i-1}(q_i\dots  q_kp_1)p_2\dots p_m$$
is a minimal r-factorisation if $i<k$, while if $i=k$ there is a minimal r-factorisation  of $zw_1$ of the form 
$$zw_1=q_1\dots q_{k-1}(q_kp_1\dots p_j) p_{j+1}\dots p_m$$ 
for some $1\leq j\leq m$ (here $j\geq 1$ is the maximal index with the property that $\ol{q_kp_1\dots p_j}$ is regular, 
which exists, as we already know in this subcase that $\ol{q_kp_1}$ is regular). According to our previous analysis, there is only 
one way in which this could fail: that $\ol{q_l\dots q_i\dots q_kp_1}$ is regular for some $l<i$ (with a seed lying within 
$q_l\dots q_kp_1$). However, this is impossible, as $\ol{q_l\dots q_ke_1}\;\R\;\ol{q_l\dots q_kp_1}$ implies that $\ol{q_l\dots q_ke_1}$ 
is regular, a contradiction. Summing up, we again have that both $zw_1$ and $ze_1$ have reduced form minimal r-factorisations whose 
first factors are $q_1$, and so the required conclusion follows just as in Case (1).

\emph{Case (3):} $\ol{ze_1}$ is regular, so that Lemma \ref{lem:FG} can be applied to the entire product $q_1\dots q_k e_1$. 
Now we have 
$$\ol{ze_1}=\ol{q_1\dots q_ke_1}\;\R\;\ol{q_1\dots q_kp_1},$$ 
just as in Case (2), so $\ol{q_1\dots q_kp_1}$ is a regular element of $\ig{\cE}$. Suppose first that $k>1$. Then the seed
of the product $q_1\dots q_k e_1$ is $e_1$, and it immediately follows that 
$$zw_1=(q_1\dots q_kp_1)p_2\dots p_m$$ 
is a minimal r-factorisation, whose reduced form is obtained by applying Lemma \ref{lem:FG} to the prefix $q_1\dots q_kp_1$. 
However, since $\ol{ze_1}\;\R\;\ol{q_1\dots q_kp_1}$, it follows that no matter how we write $\ol{ze_1}=\ol{f_1}\dots$ and 
$\ol{q_1\dots q_kp_1}=\ol{f_2}\dots$ as products of idempotents from $D_{\ol{ze_1}}$, we certainly must have $\ol{f_1}\;\R\;\ol{f_2}$. 
We do know that $\ol{ze_1}\;\tR\;\ol{f_1}$ and $\ol{zw_1}\;\tR\;\ol{f_2}$, thus we arrive at the required conclusion 
$\ol{zw_1}\;\tR\;\ol{ze_1}$. It remains to consider the case $k=1$. Now there are no restrictions regarding the position of the
seed in $q_1e_1$ (it can be either $e_1$ or lying in $q_1$), but we do know that $\ol{q_1p_1}$ is regular. Hence, it follows 
(similarly as in Case (2)) that there is a minimal r-factorisation of the form 
$$zw_1=(q_1p_1\dots p_j)p_{j+1}\dots p_m,$$
where $j\geq 1$ is the maximal index with the property that $\ol{q_1p_1\dots p_j}$ is regular. Since there is a seed of 
$q_1p_1\dots p_j$ lying within one of $q_1,p_1$, we have 
$$\ol{ze_1}=\ol{q_1e_1}\;\R\;\ol{q_1p_1}\;\R\;\ol{q_1p_1\dots p_j},$$ 
and now $\ol{zw_1}\;\tR\;\ol{ze_1}$ follows by an analogous argument as above.

This completes the proof that $\tR$ is a left congruence of $\ig{\cE}$; by left-right duality we can show that $\tL$ is a right
congruence. Therefore, the Fountain semigroup $\ig{\cE}$ satisfies the congruence condition.
\end{proof}

It was proved in \cite{YG} that $\ig{\cE}$ is abundant whenever $\cE$ is the biordered set of a semilattice. On the other hand,
Example 6.5 of the same paper supplies an example of a finite biordered set $\cE$ such that $\ig{\cE}$ is not abundant. Hence,
we finish the paper by posing the following intriguing question.

\begin{prb}
For which finite biordered sets $\cE$ is $\ig{\cE}$ abundant?
\end{prb}

\noindent\textbf{Acknowledgements.}
We are grateful to the anonymous referee for an exceptionally thorough and insightful reading of the manuscript, resulting in 
a number of comments and suggestions (and several short-cuts in proofs) which improved the quality of the paper. The first- and
the second-named author acknowledge the hospitality of the Department of Mathematics of the University of York, UK, where
a significant part of this research have been carried out.


\end{document}